\documentclass{article}
\usepackage[hmargin=1in,vmargin=1in]{geometry}

\usepackage{amsmath, amsthm, amssymb, verbatim}
\usepackage{color}
\usepackage{pstricks}
\usepackage[colorlinks]{hyperref}
\usepackage{enumerate}
\usepackage[showonlyrefs]{mathtools}
\mathtoolsset{showonlyrefs}

\newtheorem{theorem}{Theorem}[section]
\newtheorem{definition}[theorem]{Definition}
\newtheorem{lemma}[theorem]{Lemma}

\newtheorem{conjecture}[theorem]{Conjecture}

\renewcommand{\P}{\operatorname*{\mathbf{P}}}
\newcommand{\E}{\operatorname*{\mathbf{E}}}
\newcommand{\Var}{\operatorname*{\mathbf{Var}}}
\newcommand{\Cov}{\operatorname{\mathbf{Cov}}}

\DeclareMathOperator{\VAL}{VAL}
\DeclareMathOperator{\Inf}{Inf}
\DeclareMathOperator{\Norm}{N}
\DeclareMathOperator{\PLUR}{PLUR}
\DeclareMathOperator{\MAJ}{MAJ}
\DeclareMathOperator{\UniqueBest}{UniqueBest}
\newcommand{\ip}[2]{\langle #1 , #2 \rangle}
\newcommand{\bigoh}{\ensuremath{\mathcal{O}}}
\newcommand{\calV}{\ensuremath{\mathcal{V}}}
\newcommand{\calX}{\ensuremath{\mathcal{X}}}

\newcommand{\calY}{\ensuremath{\mathcal{Y}}}
\newcommand{\calG}{\ensuremath{\mathcal{G}}}
\newcommand{\calH}{\ensuremath{\mathcal{H}}}
\newcommand{\calL}{\ensuremath{\mathcal{L}}}
\newcommand{\calM}{\ensuremath{\mathcal{M}}}
\newcommand{\Reals}{\ensuremath{\mathbb{R}}}
\newcommand{\Borel}{\ensuremath{\mathbb{B}}}
\newcommand{\calZ}{\ensuremath{\mathcal{Z}}}
\DeclareMathOperator{\stab}{\mathbb{S}}
\DeclareMathOperator{\Sp}{\stab_{\rho}}

\newcommand{\Qtilde}{\ensuremath{\widetilde{Q}}}
\renewcommand{\vec}[1] {\ensuremath{\mathbf{#1}}}

\newcommand{\Infd}{\Inf^{\le d}}
\newcommand{\Sphere}{\ensuremath{\mathrm{S}}}

\newcommand{\SphereMo}{\ensuremath{\Sphere^{m\!-\!1}}}

\definecolor{Red}{rgb}{1,0,0}
\definecolor{Blue}{rgb}{0,0,1}
\definecolor{Olive}{rgb}{0.41,0.55,0.13}
\definecolor{Green}{rgb}{0,1,0}
\definecolor{MGreen}{rgb}{0,0.8,0}
\definecolor{DGreen}{rgb}{0,0.55,0}
\definecolor{Yellow}{rgb}{1,1,0}
\definecolor{Cyan}{rgb}{0,1,1}
\definecolor{Magenta}{rgb}{1,0,1}
\definecolor{Orange}{rgb}{1,.5,0}
\definecolor{Violet}{rgb}{.5,0,.5}
\definecolor{Purple}{rgb}{.75,0,.25}
\definecolor{Brown}{rgb}{.75,.5,.25}
\definecolor{Grey}{rgb}{.5,.5,.5}
\definecolor{Black}{rgb}{0,0,0}

\newcommand{\eps}{\epsilon}

\begin{document}

\title{Maximally Stable Gaussian Partitions with Discrete Applications}
\author{
  Marcus Isaksson
  \thanks{
    Chalmers University of Technology and G{\"o}teborg University,
    SE-41296 G{\"o}teborg, Sweden.
    maris@chalmers.se.
  }
  \and
  Elchanan Mossel
  \thanks{U.C. Berkeley and Weizmann Institute.
    mossel@stat.berkeley.edu.
    Supported by an Alfred Sloan fellowship
    in Mathematics, by NSF CAREER grant DMS-0548249 (CAREER), by DOD ONR
    grant (N0014-07-1-05-06), by BSF grant 2004105 and by ISF grant
    1300/08
  }
}

\maketitle

\begin{abstract}
Gaussian noise stability results have recently played an important role
in proving results in hardness of approximation in
computer science and in the study of voting schemes in social
choice.
We prove a new Gaussian noise stability result
generalizing an isoperimetric result by Borell on the heat kernel
and derive as applications:
\begin{itemize}
\item
An optimality result for majority in the context of Condorcet voting.
\item
A proof of a conjecture on ``cosmic coin tossing'' for low influence
functions.
\end{itemize}
We also discuss a Gaussian noise stability conjecture which may be
viewed as a generalization of the ``Double Bubble'' theorem and show
that it implies:
\begin{itemize}
\item
A proof of the ``Plurality is Stablest Conjecture''.
\item
That the Frieze-Jerrum SDP for MAX-q-CUT achieves the optimal
approximation factor assuming the Unique Games Conjecture.
\end{itemize}

\end{abstract}

\section{Introduction}

Recent results in hardness of approximation in computer science
\cite{KKMO:07,DiMoRe:06,Austrin:07a,Austrin:07b,
  ODonnellYi:08,Guruswami:08}
and in the study of voting schemes in social choice
\cite{Kalai:02,Mossel:08}
crucially rely
on Gaussian noise stability results. The first result in hardness of
approximation established a tight inapproximability result for
MAX-CUT assuming the Unique Games Conjecture~\cite{KKMO:04}
while the latest results
conditionally achieve optimal
inapproximation factors for very general families of
constraint satisfaction problems~\cite{Austrin:07b,Raghavendra:08}.
Results in social choice include optimality of the majority function among low
influence functions in the context of Condorcet voting on $3$
candidates~\cite{Kalai:02} and near optimality
for any number of candidates~\cite{Mossel:08}.
A common feature of these results is the use of ``Invariance
Principles''~\cite{MoOdOl:05,MoOdOl:09,Mossel:08} together with an optimal Gaussian
noise stability result by Borell~\cite{Borell:85}.

In the current paper we prove a theorem
generalizing the result of Borell~\cite{Borell:85},
discuss a related conjecture and develop an
extension of the invariance principle.
As applications we derive some new results in
hardness of approximation and social choice.
In the introduction we state the theorem and the conjecture together
with their applications and provide ``moral support'' for the
correctness of the conjecture.

\subsection{Maximally Stable Gaussian Partitions}
We will be concerned with finding partitions of $\Reals^n$
that maximize the probability that correlated Gaussian vectors
remain within the same part.
More specifically we would like to partition $\Reals^n$ into $q\ge 2$
disjoint sets of
fixed measure.

Borell~\cite{Borell:85} proved that when $q=2$ and we have two
standard Gaussian vectors
with covariance $\rho \ge 0$ in corresponding coordinates then
half-spaces
are optimal.
Let $I_n$ be the $n \times n$ identity matrix. For two $n$-dimensional
random variables $X = (X_1,\ldots,X_n)$ and
$Y = (Y_1,\ldots,Y_n)$ write
$\Cov(X,Y)$ for the $n \times n$ matrix whose $(i,j)$'th entry is
given by $\Cov[X_i,Y_j] = \E[X_i Y_j] - \E[X_i] \E[Y_j]$.
Recall that $X \sim N(0,I_n)$ means that
$X$ is a standard $n$-dimensional Gaussian vector with independent entries, all of which are standard normal random variables.
Borell's result states the following:

\begin{theorem}\cite{Borell:85}
  \label{thm:Borell}
  Fix $\rho \in [0,1]$.
  Suppose $X,Y \sim \Norm(0, I_n)$ are jointly normal
  and $\Cov(X,Y)=\rho I_n$.
  Then for any
  $A_1,A_2 \subseteq \Reals^n$,
  \begin{equation}
    \P(X \in A_1, Y \in A_2) \le \P(X \in H_1, Y \in H_2)
  \end{equation}
  where $H_i=\{x \in \Reals^n | x_1 \le a_i\}$ for
  $a_i$ chosen so that $\P(X \in H_i)=\P(X \in A_i)$.
\end{theorem}

We will consider two different generalizations of
Theorem~\ref{thm:Borell}.
The first generalization claims that half-spaces are still optimal if we
have $k>2$ correlated vectors and seek to maximize the probability
that they all fall into the same part:
\begin{theorem}[Exchangeable Gaussians Theorem, or EGT]
  \label{thm:nGauss2SplitIsop}
  Fix $\rho \in [0,1]$.
  Suppose $X_1,\dots,X_k \sim \Norm(0, I_n)$
  are jointly normal and $\Cov(X_i,X_j)=\rho
  I_n$ for $i\neq j$.
  Then, for any $A_1, \ldots, A_n \subseteq \Reals^n$,
  \begin{equation}
    \label{eq:nGauss2SplitIsop}
    \P(\forall i: X_i \in A_i)
    \le
    \P(\forall i: X_i \in H_i)
  \end{equation}
  where $H_i=\{x \in \Reals^n | x_1 \le a_i\}$ for
  $a_i$ chosen so that $\P(X \in H_i)=\P(X \in A_i)$.
\end{theorem}

We call the theorem above the {\em Exchangeable Gaussians Theorem} (EGT).
Recall that a collection of random variables is exchangeable if its
distribution is invariant under any permutation.

The second generalization of Theorem \ref{thm:Borell} concerns
the optimal partition of $\Reals^n$ into $q>2$ sets.
We conjecture that
when the partition is balanced (i.e. all $q$ sets have equal measure)
the optimal partition
is a \emph{standard simplex partition} dividing $\Reals^n$ into $q$ partitions
depending on which of $q$ maximally separated unit vectors are
closest (ties may be broken arbitrarily),
and further that this is the \emph{least stable} partition for $\rho
\le 0$, even for unbalanced partitions.

\begin{definition}
  \label{def:standardSimplex}
  For $n\!+\!1 \ge q \ge 2$,
  $A_1, \ldots, A_q$ is a \emph{standard simplex partition} of $\Reals^n$ if
  for all $i$
  \begin{equation}
    A_i \supseteq \{x \in \Reals^n | x \cdot a_i > x \cdot a_j, \forall j
    \neq i\}
  \end{equation}
  where $a_1, \ldots a_q \in \Reals^n$ are $q$ vectors satisfying
  \begin{equation}
    \label{eq:stdSimplexVecConstr}
    a_i \cdot a_j =
    \left\{
      \begin{array}{ll}
        1 & \text{ if } i=j \\
        -\frac{1}{q-1} & \text{ if } i \neq j
      \end{array}
    \right.
  \end{equation}
\end{definition}
When $n \ge q$ a standard simplex partition can be formed by picking
$q$ orthonormal vectors $e_1, \ldots, e_q$, subtracting their mean and
scaling appropriately, i.e.
\begin{equation}
  a_i = \sqrt{\frac{q}{q-1}} \left(e_i - \frac{1}{q} \sum_{i=1}^q e_i\right)
\end{equation}
and for $n = q-1$ it is enough to project these vectors onto the
$q-1$-dimensional space which they span.

We call $A_1, \ldots, A_q$ a balanced partition of $\Reals^n$ if
$A_1, \ldots, A_q$ are disjoint with $\P(X \in A_j)=\frac 1 q$, $\forall j$.

\begin{conjecture}[Standard Simplex Conjecture, or SSC]
  \label{conj:2GaussQSplitIsop}
  Fix $\rho \in [-1,1]$ and $3 \le q \le n \! + \! 1$.
  Suppose $X,Y \sim \Norm(0, I_n)$ are jointly normal and $\Cov(X,Y)=\rho I_n$.
  Let $A_1,\dots,A_q \subseteq \Reals^n$
  be a partition of $\Reals^n$
  and $S_1, \dots,S_q \subseteq \Reals^n$ a standard simplex partition.
  Then,
  \begin{enumerate}[i)]
  \setlength{\itemsep}{1pt}
  \setlength{\parskip}{0pt}
  \setlength{\parsep}{0pt}
  \item
    If $\rho \ge 0$ and $A_1, \ldots, A_q$ is \emph{balanced}, then
    \begin{equation}
      \label{eq:2GaussQSplitIsop}
      \P((X,Y) \in A_1^2 \cup \dots \cup A_q^2)
      \le
      \P((X,Y) \in S_1^2 \cup \dots \cup S_q^2)
    \end{equation}
    \item
    If $\rho <0$, \eqref{eq:2GaussQSplitIsop} holds in reverse:
    \begin{equation}
      \label{eq:2GaussQSplitIsopRev}
      \P((X,Y) \in A_1^2 \cup \dots \cup A_q^2)
      \ge
      \P((X,Y) \in S_1^2 \cup \dots \cup S_q^2)
    \end{equation}
  \end{enumerate}
\end{conjecture}
Conjecture~\ref{conj:2GaussQSplitIsop} was made by the second author
around 2004 during the writing of~\cite{MoOdOl:05}.

When $q=3$ the standard simplex partition,
also known as the \emph{standard Y partition} or the
\emph{peace sign partition},
is described in $\Reals^2$ by three
half-lines meeting at an $120$ degree angle at the origin
(Figure \ref{fig:peaceSign})
and in $\Reals^n$, where $n>2$, it can be exemplified by taking
the Cartesian product of the
peace sign partition and $\Reals^{n-2}$.

\begin{figure}[h]
  \label{fig:peaceSign}
  \centering
\psset{xunit=.5pt,yunit=.5pt,runit=.5pt}
\begin{pspicture}(104.37103271,96.30863953)
{
\newrgbcolor{curcolor}{0 0 0}
\pscustom[linewidth=0.35433072,linecolor=curcolor,linestyle=dashed,dash=2.12598425 0.35433071]
{
\newpath
\moveto(51.93925,-7.17577047)
\curveto(51.93925,63.69036953)(51.93925,63.69036953)(51.93925,63.69036953)
}
}
{
\newrgbcolor{curcolor}{0 0 0}
\pscustom[linestyle=none,fillstyle=solid,fillcolor=curcolor]
{
\newpath
\moveto(52.79681578,-5.32221239)
\lineto(51.94265439,-7.64503676)
\lineto(51.0884932,-5.32221226)
\curveto(51.59278963,-5.69330248)(52.28275466,-5.69116425)(52.79681578,-5.32221239)
\closepath
}
}
{
\newrgbcolor{curcolor}{0 0 0}
\pscustom[linewidth=0.35433072,linecolor=curcolor,linestyle=dashed,dash=2.12598425 0.35433071]
{
\newpath
\moveto(87.37233,28.25729953)
\curveto(16.50618,28.25729953)(16.50618,28.25729953)(16.50618,28.25729953)
}
}
{
\newrgbcolor{curcolor}{0 0 0}
\pscustom[linestyle=none,fillstyle=solid,fillcolor=curcolor]
{
\newpath
\moveto(85.51877191,29.1148653)
\lineto(87.84159628,28.26070392)
\lineto(85.51877178,27.40654273)
\curveto(85.88986201,27.91083915)(85.88772378,28.60080419)(85.51877191,29.1148653)
\closepath
}
}
{
\newrgbcolor{curcolor}{0 0 0}
\pscustom[linewidth=1.70787406,linecolor=curcolor]
{
\newpath
\moveto(94.89956,3.45417953)
\lineto(70.63653,17.34169953)
\lineto(52.43925,27.75732953)
\lineto(52.43925,27.75732953)
}
}
{
\newrgbcolor{curcolor}{0 0 0}
\pscustom[linestyle=none,fillstyle=solid,fillcolor=curcolor]
{
\newpath
\moveto(108.5362241,-4.35108822)
\curveto(108.91087696,-3.69652833)(109.74618006,-3.46935741)(110.40073995,-3.84401027)
\curveto(111.05529984,-4.21866313)(111.28247076,-5.05396622)(110.9078179,-5.70852611)
\curveto(110.53316504,-6.363086)(109.69786195,-6.59025692)(109.04330206,-6.21560406)
\curveto(108.38874217,-5.84095121)(108.16157125,-5.00564811)(108.5362241,-4.35108822)
\closepath
}
}
{
\newrgbcolor{curcolor}{0 0 0}
\pscustom[linestyle=none,fillstyle=solid,fillcolor=curcolor]
{
\newpath
\moveto(103.34836275,-1.38169288)
\curveto(103.72301561,-0.72713299)(104.55831871,-0.49996207)(105.2128786,-0.87461493)
\curveto(105.86743848,-1.24926778)(106.09460941,-2.08457088)(105.71995655,-2.73913077)
\curveto(105.34530369,-3.39369066)(104.51000059,-3.62086158)(103.85544071,-3.24620872)
\curveto(103.20088082,-2.87155586)(102.9737099,-2.03625277)(103.34836275,-1.38169288)
\closepath
}
}
{
\newrgbcolor{curcolor}{0 0 0}
\pscustom[linestyle=none,fillstyle=solid,fillcolor=curcolor]
{
\newpath
\moveto(98.1605014,1.58770246)
\curveto(98.53515426,2.24226235)(99.37045736,2.46943327)(100.02501725,2.09478042)
\curveto(100.67957713,1.72012756)(100.90674806,0.88482446)(100.5320952,0.23026457)
\curveto(100.15744234,-0.42429531)(99.32213924,-0.65146624)(98.66757936,-0.27681338)
\curveto(98.01301947,0.09783948)(97.78584855,0.93314258)(98.1605014,1.58770246)
\closepath
}
}
{
\newrgbcolor{curcolor}{0 0 0}
\pscustom[linewidth=1.70787406,linecolor=curcolor]
{
\newpath
\moveto(9.47894,3.95414953)
\lineto(33.74197,17.84166953)
\lineto(51.93925,28.25729953)
\lineto(51.93925,28.25729953)
}
}
{
\newrgbcolor{curcolor}{0 0 0}
\pscustom[linestyle=none,fillstyle=solid,fillcolor=curcolor]
{
\newpath
\moveto(-4.1577241,-3.85111822)
\curveto(-3.78307125,-4.50567811)(-4.01024217,-5.34098121)(-4.66480206,-5.71563406)
\curveto(-5.31936195,-6.09028692)(-6.15466504,-5.863116)(-6.5293179,-5.20855611)
\curveto(-6.90397076,-4.55399622)(-6.67679984,-3.71869313)(-6.02223995,-3.34404027)
\curveto(-5.36768006,-2.96938741)(-4.53237696,-3.19655833)(-4.1577241,-3.85111822)
\closepath
}
}
{
\newrgbcolor{curcolor}{0 0 0}
\pscustom[linestyle=none,fillstyle=solid,fillcolor=curcolor]
{
\newpath
\moveto(1.03013725,-0.88172288)
\curveto(1.4047901,-1.53628277)(1.17761918,-2.37158586)(0.52305929,-2.74623872)
\curveto(-0.13150059,-3.12089158)(-0.96680369,-2.89372066)(-1.34145655,-2.23916077)
\curveto(-1.71610941,-1.58460088)(-1.48893848,-0.74929778)(-0.8343786,-0.37464493)
\curveto(-0.17981871,0.00000793)(0.65548439,-0.22716299)(1.03013725,-0.88172288)
\closepath
}
}
{
\newrgbcolor{curcolor}{0 0 0}
\pscustom[linestyle=none,fillstyle=solid,fillcolor=curcolor]
{
\newpath
\moveto(6.2179986,2.08767246)
\curveto(6.59265145,1.43311258)(6.36548053,0.59780948)(5.71092064,0.22315662)
\curveto(5.05636076,-0.15149624)(4.22105766,0.07567469)(3.8464048,0.73023457)
\curveto(3.47175194,1.38479446)(3.69892287,2.22009756)(4.35348275,2.59475042)
\curveto(5.00804264,2.96940327)(5.84334574,2.74223235)(6.2179986,2.08767246)
\closepath
}
}
{
\newrgbcolor{curcolor}{0 0 0}
\pscustom[linewidth=1.70787406,linecolor=curcolor]
{
\newpath
\moveto(51.93925,28.25729953)
\lineto(51.93925,77.86359953)
}
}
{
\newrgbcolor{curcolor}{0 0 0}
\pscustom[linestyle=none,fillstyle=solid,fillcolor=curcolor]
{
\newpath
\moveto(51.93925,84.35352097)
\curveto(52.6934472,84.35352097)(53.30554927,83.7414189)(53.30554927,82.98722171)
\curveto(53.30554927,82.23302451)(52.6934472,81.62092244)(51.93925,81.62092244)
\curveto(51.1850528,81.62092244)(50.57295073,82.23302451)(50.57295073,82.98722171)
\curveto(50.57295073,83.7414189)(51.1850528,84.35352097)(51.93925,84.35352097)
\closepath
}
}
{
\newrgbcolor{curcolor}{0 0 0}
\pscustom[linestyle=none,fillstyle=solid,fillcolor=curcolor]
{
\newpath
\moveto(51.93925,90.33108018)
\curveto(52.6934472,90.33108018)(53.30554927,89.71897811)(53.30554927,88.96478091)
\curveto(53.30554927,88.21058372)(52.6934472,87.59848165)(51.93925,87.59848165)
\curveto(51.1850528,87.59848165)(50.57295073,88.21058372)(50.57295073,88.96478091)
\curveto(50.57295073,89.71897811)(51.1850528,90.33108018)(51.93925,90.33108018)
\closepath
}
}
{
\newrgbcolor{curcolor}{0 0 0}
\pscustom[linestyle=none,fillstyle=solid,fillcolor=curcolor]
{
\newpath
\moveto(51.93925,96.30863939)
\curveto(52.6934472,96.30863939)(53.30554927,95.69653732)(53.30554927,94.94234012)
\curveto(53.30554927,94.18814293)(52.6934472,93.57604086)(51.93925,93.57604086)
\curveto(51.1850528,93.57604086)(50.57295073,94.18814293)(50.57295073,94.94234012)
\curveto(50.57295073,95.69653732)(51.1850528,96.30863939)(51.93925,96.30863939)
\closepath
}
}
\end{pspicture}
  \caption{The peace sign partition}
\end{figure}
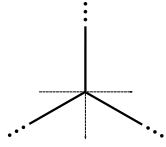

\subsection{Applications}

Given the numerous applications of the results of Borell together with invariance
\cite{KKMO:07,DiMoRe:06,Austrin:07a,Austrin:07b,
  ODonnellYi:08,Guruswami:08,
  Kalai:02,Mossel:08},
it is natural to expect that the generalizations discussed here will have a wide variety of applications. Here we derive the first applications in social choice theory and hardness of approximation in computer science:
\begin{itemize}
\item
  From the EGT we derive
  certain optimality of majority in Condorcet voting.
  More specifically, majority asymptotically maximizes the probability of
  having a unique winner in Condorcet voting with any number of
  candidates among low influence voting schemes.
  It also maximizes the probability of $k$ players agreeing, among low-influence functions, in the setting of cosmic coin flipping
  ~\cite{MosselODonnell:05, MORSS:06}.

\item
  The SSC implies
  the \emph{Plurality is Stablest conjecture} as well as showing that the
  Frieze-Jerrum~\cite{FriezeJerrum:95} SDP relaxation obtains the
  optimal approximation ratio for MAX-q-CUT assuming the Unique Games
  Conjecture.
\end{itemize}

The main tool for proving these applications is the invariance
principle of~\cite{MoOdOl:05,MoOdOl:09,Mossel:08} which we extend to
handle general Lipschitz continuous functions. We note that previous work proved
the invariance principle for $\mathcal{C}^3$ functions and some
specific Lipschitz continuous functions. The generalization of the invariance principle may be of independent interest.
We proceed with formal statements of the applications.

\subsubsection{Condorcet voting}
\label{sec:introCondorcet}
Suppose $n$ voters rank $k$ candidates by each voter $i$
providing a linear order $\sigma_i \in S(k)$ on the candidates.
In {\em Condorcet voting},
the rankings are aggregated by deciding for each pair
of candidates which one is preferred over the other by the $n$ voters.
This decision can be performed in many ways, but we will require that it
satisfies two criteria
\begin{itemize}
\item
  \emph{Independence of Irrelevant Alternatives (IIA)}.
  The decision of whether $a$ is preferred over $b$ can only depend on each voter's preference between $a$ and $b$.
\item
  \emph{Neutrality}. The decision must be invariant under permutations on the $k$ candidates.
\end{itemize}

More formally, the aggregation results in a tournament $G_k$ on the set $[k]$.
Recall that $G_k$ is a {\em tournament} on $[k]$ if it is a
directed graph on the vertex set $[k]$ such that for all $a,b \in [k]$
either $(a > b) \in G_k$ or $(b > a) \in G_k$.
Given individual
rankings $(\sigma_i)_{i=1}^n$ the tournament $G_k$ is defined as follows.
Let
\begin{equation}
 x_i^{a>b}=
 \left\{
   \begin{array}{ll}
     1 & \mbox{ if } \sigma_i(a) > \sigma_i(b) \\
     -1 & \mbox{ else }
   \end{array}
 \right.
 \text{, for } i \in [n] \text{ and } a,b \in [k].
\end{equation}
Note that $x^{b > a } = -x^{a > b}$.
By IIA and neutrality we may assume that
the binary decision between each pair of candidates is performed via
an anti-symmetric function $f : \{-1,1\}^n \to \{0,1\}$ so that
$f(-x) = 1-f(x)$ for all $x \in \{-1,1\}^n$.
The tournament $G_k = G_k(\sigma; f)$ is then defined
by letting $(a > b) \in G_k$ if and only if $f(x^{a > b}) = 1$.
A natural decision function is the majority function
$\MAJ_n:\{-1,1\}^n \rightarrow \{0,1\}$
defined by
$\MAJ_n(x) = 1_{\{\sum_{i=1}^n x_i \ge 0\}}$.

For the purposes of
social choice, some tournaments make more sense than others.
For example, it is desirable to have $G_k$ provide a linear ranking
of the candidates. But this is more than we can hope for,
since by Arrows impossibility theorem~\cite{Arrow:50} a linear ranking can
only be guaranteed when $f$ is a dictator, i.e.
$f(x) = \frac{1 \pm x_i}{2}$.
A weaker requirement is that there exist a
\emph{Condorcet winner} or a \emph{unique best candidate} in $G_k$,
i.e. for some $a \in [k]$: $(a > b) \in G_k, \forall b \neq a$.

Following~\cite{Kalai:04,Kalai:02,Mossel:08},
we consider the probability distribution
over $n$ voters, where the voters have independent preferences and
each one chooses a ranking uniformly at random among all $k!$
orderings. Note that the marginal
distributions on vectors $x^{a > b}$ is the uniform distribution over
$\{-1,1\}^n$ and that if $f : \{-1,1\}^n \to \{0,1\}$ is
anti-symmetric then $\E[f] = \frac{1}{2}$. The previous discussion and
the following definition are essentially taken from~\cite{Mossel:08}.
\begin{definition}
\label{def:uniqueBest}
For any anti-symmetric $f:\{-1,1\}^n \rightarrow \{0,1\}$
  let $\UniqueBest_k(f)$ denote the event that the Condorcet voting system
  described above results in a unique best candidate
  and $\UniqueBest_k(f,i)$ the event that the i:th candidate is unique best.
\end{definition}

Consider first the case $k=3$. In
this case $G_3$ has a unique best candidate if and only if it
corresponds to a linear ranking of the candidates.
Kalai~\cite{Kalai:02} studied the \emph{probability} of a rational
outcome (i.e. $G_3$ linear)
given that the $n$ voters vote independently and at random
from the $6$ possible rational rankings.  He showed that the
probability of a rational outcome in this case may be expressed as
$3 \, \stab_{1/3}(f)$.
Here $\Sp(f)$ denotes the noise stability
$\Sp(f) = \E[f(X)f(Y)]$
where $X$ is uniform on $\{-1,1\}^n$
and $Y$ is obtained from $X$
by independently rerandomizing each coordinate with probability $1-\rho$.

The influence $\Inf_i f$ of voter $i$, is the probability that voter $i$ can change the outcome of the election, i.e.
$\Inf_i f = \P(f(X) \neq f(X^{(i)}))$ where $X^{(i)}$ is obtained from $X$ by flipping the $i$:th coordinate.
It is natural to ask which function
$f$ with small influences is most likely to produce a rational
outcome.  Instead of considering small influences, Kalai
considered the essentially stronger assumption that $f$ is monotone and
``transitive-symmetric''; i.e., that for all $1 \leq i < j \leq n$
there exists a permutation $\sigma$ on $[n]$ with $\sigma(i) = j$
such that $f(x_1, \dots, x_n) = f(x_{\sigma(1)}, \dots,
x_{\sigma(n)})$ for all $(x_1, \dots, x_n)$.
Kalai conjectured that as $n \to \infty$ the maximum of
$3 \stab_{1/3}(f)$ among all transitive-symmetric functions
approaches the same limit as
$\lim_{n \to \infty} 3\, \stab_{1/3}(\MAJ_n)$.
This follows directly from
the Majority is Stablest Theorem~\cite{MoOdOl:05,MoOdOl:09}.
In~\cite{Mossel:08} similar, but sub-optimal results were obtained for any value of $k$.
More specifically it was shown that if
one considers Condorcet voting on $k$ candidates, then for all $\eps > 0$ there
exists $\tau  = \tau(k,\eps) > 0$
such that if $f : \{-1,1\}^n \to \{0,1\}$ is anti-symmetric
and $\Inf_i(f) \leq \tau$ for all $i$, then
\begin{equation} \label{eq:unique_max}
\P[\UniqueBest_k(f)] \leq k^{-1+o_k(1)} + \eps.
\end{equation}
Moreover for the majority function we have
$\Inf_i(\MAJ_n) = O(n^{-1/2})$
and it holds that
\begin{equation} \label{eq:maj_unique_max}
\P[\UniqueBest_k(\MAJ_n)] \geq k^{-1-o_k(1)} - o_n(1).
\end{equation}

\noindent
As a consequence of the EGT we provide tight results for every value of $k$.
\begin{theorem}
  \label{thm:condorcet}
  For any $k \ge 1$ and $\epsilon>0$
  there exists a $\tau(\epsilon, k)>0$
  such that for any anti-symmetric $f:\{-1,1\}^n \rightarrow \{0,1\}$
  satisfying $\max_i \Inf_i f \le \tau$,
  \begin{equation}
    \P[\UniqueBest_k(f)]
    \le
    \lim_{n \rightarrow \infty}
    \P[\UniqueBest_k(\MAJ_n)] + \epsilon
  \end{equation}
\end{theorem}

\subsubsection{Cosmic coin flipping}
In the setting of \emph{cosmic coin flipping}
as studied in~\cite{MosselODonnell:05, MORSS:06},
we have $k$ players
and a source $X \in \{-1,1\}^n$ of $n$ uniform bits.
Each player $i$ is given a noisy version $Y_i$ of $X$,
where each bit in $Y_i$ is a noisy copy
of the corresponding bit in $X$.
More specifically, given $X$, each $Y_{i,j}$ is selected independently as
\begin{equation}
  Y_{i,j} =
  \left\{
    \begin{array}{rl}
      X_j & \text{ w.p. } \frac{1+\rho}{2}
      \\
      -X_j & \text{ w.p. } \frac{1-\rho}{2}
    \end{array}
  \right.
\end{equation}
Note that
$\E[X_j Y_{i,j}] = \rho$
and
$\E[Y_{i,j} Y_{l,j}] = \rho^2$ for $i \neq l$.

The $k$ players want to use their noisy versions of $X$ to flip a balanced coin such that they all agree on an outcome with maximal probability, i.e. they want to select a balanced function $f:\{-1,1\}^n \rightarrow \{0,1\}$ maximizing
\begin{equation}
  \mathcal{P}^{(k,n)}_\rho(f) := \P(f(Y_1) = \ldots = f(Y_k))
\end{equation}
The requirement that all players must use the same function should not
be considered a restriction since, as shown in~\cite{MosselODonnell:05},
allowing each player to use a different function $f_i$ cannot increase the maximal probability of all players agreeing.

The main conjecture of~\cite{MosselODonnell:05} is that
for any fixed $k$, $n$ and $\rho$,
$\mathcal{P}^{(k,n)}_\rho(f)$ is maximized by $\MAJ_m$ for some $m \le n$.
As a consequence of the EGT we show that if $f$ is required to have low influence in each coordinate, then majority asymptotically maximizes $\mathcal{P}^{(k,n)}_\rho(f)$ for unbounded $n$,
\begin{theorem}
  \label{thm:cosmicCoin}
  For any $k \ge 1$, $\rho \in [0,1]$ and $\epsilon>0$
  there exists a $\tau(\eps,k,\rho)>0$
  such that for any balanced $f:\{-1,1\}^n \rightarrow \{0,1\}$
  satisfying $\max_i \Inf_i f \le \tau$,
  \begin{equation}
    \mathcal{P}^{(k,n)}_\rho(f)
    \le
    \lim_{n \rightarrow \infty}
    \mathcal{P}^{(k,n)}_\rho(\MAJ_n) + \epsilon
  \end{equation}
\end{theorem}

\subsubsection{Plurality is Stablest}

Consider an election with $n$ voters choosing between $q$ candidates.
We call a function $f:[q]^n \rightarrow [q]$, which given the $n$ votes determines the winning candidate, {\em a social choice function}.
Letting $\Delta_q = \{x \in \Reals^q| x \ge 0, \sum_{i=1}^q x_i = 1\}$
denote the standard q-simplex,
we generalize this notion a bit and call a function $f:[q]^n
\rightarrow \Delta_q$ assigning a probability distribution to the set of candidates a {\em ``fuzzy'' social choice function.}
To be able to treat non-fuzzy social choice functions at the same time, we
will usually embed their output into $\Delta_q$ and think of them as
functions $f:[q]^n \rightarrow E_q$, where
$E_q = \{e_1, \ldots, e_q\}=\{(1,0, \dots,0), \dots, (0,\dots, 0, 1)\}$
are the $q$
extreme points of $\Delta_q$ corresponding to assigning a
probability mass $1$ to one of the candidates.

The noise stability of such functions measures the stability of the output when
the votes are chosen independently and uniformly at random, and then rerandomized with probability $1-\rho$.

\begin{definition}
  \label{def:SpQ}
  For $-\frac{1}{q-1} \le \rho \le 1$, the noise stability of
  $f:[q]^n \rightarrow \Reals^k$ is
  \begin{equation}
    \Sp (f) = \sum_{j=1}^k \E [f_j(\omega) f_j(\lambda)]
  \end{equation}
  where $\omega$ is uniformly selected from $[q]^n$ and each $\lambda_i$ is independently selected using the conditional distribution
  \begin{equation}
    \label{eq:defSpQ}
    \mu(\lambda_i|\omega_i) = \rho 1_{\{\lambda_i=\omega_i\}} + (1-\rho) \frac 1 q
  \end{equation}
\end{definition}
\noindent
Note that when $f:[q]^n \rightarrow E_q$ is a non-fuzzy social choice
function, $\Sp(f) = \P(f(\omega) = f(\lambda))$.

We say that $f:[q]^n \rightarrow \Delta_q$ is {\em balanced} if
$\E[f(\omega)]=\frac 1 q \vec{1}$ where $\omega$ is uniformly selected
from $[q]^n$ and
say that the influence of the $i$:th coordinate on
a real valued function $f:[q]^n \rightarrow
\Reals$ is
\begin{equation}
  \label{eq:influence}
  \Inf_i f(\omega)
  =
  \E_\omega [
  \Var_{\omega_i} f(\omega)
  ]
\end{equation}

Note that the definition of $\Sp$ generalizes the definition for Boolean function in Section~\ref{sec:introCondorcet} if we identify $[2]^n$ with $\{-1,1\}^n$.
Likewise, the definition of influence here generalizes the notion of influence
in Section~\ref{sec:introCondorcet} except that the earlier notion is precisely $4$ times larger. This distinction is usually not important, but when it is we will use the latter one given by \eqref{eq:influence}.

Let
$\PLUR_{n,q}:[q]^n \rightarrow \Delta_q$
denote the plurality function which assigns a probability mass $1$ to the candidate with the most votes (ties can be broken arbitrarily, e.g. by splitting the mass equally among the tied candidates).
The \textit{Plurality is Stablest} conjecture claims that plurality is essentially the most stable of all low-influence functions under uniform measure:

\begin{conjecture}[Plurality is Stablest]
  \label{conj:plurStabAppl}
  For any $q \ge 2$, $\rho \in [-\frac{1}{q-1},1]$
  and $\epsilon>0$ there exists a $\tau>0$ such that
  if $f:[q]^n \rightarrow \Delta_q$
  has $\Inf_i(f_j) \le \tau$, $\forall i,j$,
  then
  \begin{align}
    \label{eq:plurStabPos}
    &\Sp(f) \le \lim_{n \rightarrow \infty} \Sp(\PLUR_{n,q})
    + \epsilon \hspace{0.4cm} \mbox { if } \rho \ge 0
    \mbox{ and } f \mbox{ is balanced }
    \shortintertext{and}
    & \Sp(f) \ge \lim_{n \rightarrow \infty} \Sp(\PLUR_{n,q})
    - \epsilon  \hspace{0.4cm} \mbox { if } \rho \le 0
    \label{eq:plurStabNeg}
  \end{align}
\end{conjecture}

The case where $q=2$, the \emph{Majority is stablest theorem}, was
proved in~\cite{MoOdOl:09}.
The question if Plurality is Stablest or not has been discussed by Khot, Kindler, O'Donnell and the second author as part of their work~\cite{KKMO:07}.
Here we conjecture that Plurality is indeed stablest.

We show that the general case follows from SSC.

\begin{theorem}
  \label{thm:plurStabAppl}
  SSC (Conj. \ref{conj:2GaussQSplitIsop}) $\Rightarrow$
  Plurality is Stablest (Conj.  \ref{conj:plurStabAppl})
\end{theorem}
It should be pointed out that our results imply a slightly stronger
result where the small influence requirement is replaced by a small
\emph{low-degree influence} requirement. This strengthening turns out to be
crucial to applications in hardness of approximation. 

We also show the reverse implication for $\rho \ge -\frac{1}{q-1}$,
implying that the Plurality is Stablest conjecture is equivalent to
the SSC for $\rho$ in this range.
\begin{theorem}
  \label{thm:plurStabApplRev}
  Plurality is Stablest (Conj.  \ref{conj:plurStabAppl})
  $\Rightarrow$
  SSC (Conj. \ref{conj:2GaussQSplitIsop})
  for $\rho \in [-\frac{1}{q-1},1]$
\end{theorem}

It follows from calculations in~\cite{KKMO:07} that the bound
\eqref{eq:plurStabNeg} in Conjecture~\ref{conj:plurStabAppl}
holds asymptotically for $\rho = -\frac{1}{q-1}$
as $q \rightarrow \infty$, i.e.
\begin{equation}
  \label{eq:plurStabAsympt}
  \stab_{-\frac{1}{q-1}}(f)
  \ge
  (1 - o_q(1)) \cdot \lim_{n \rightarrow \infty} \Sp(\PLUR_{n,q})
  + \epsilon_q (\tau)
\end{equation}
where $\epsilon_q(\tau) \rightarrow 0$ as $\tau \rightarrow 0$.

It may be helpful to think of the theorem in terms of a
pure social choice function $f: [q]^n \to [q]$. In this case, there
are $n$ voters and each voter chooses one out of $q$ possible
candidates.
Given individual choices $x_1,\ldots,x_n$,
the winning candidate is defined to be $f(x_1,\ldots,x_n)$. In
social choice theory it is natural to restrict attention to the class
of low influence functions, where each individual voter has small
effect on the outcome. We now consider the scenario where
voters have independent and uniform preferences. Moreover, we assume
that there is a problem with the voting machines so that each vote
cast is rerandomized with probability $1-\rho$. Denoting by
$X_1,\ldots,X_n$ the intended votes and $Y_1,\ldots,Y_n$ the
registered votes, it is natural to wonder how correlated are
$f(X_1,\ldots,X_n)$ and $f(Y_1,\ldots,Y_n)$.
Theorem~\ref{thm:plurStabAppl} states
that under SSC, the maximal amount of correlation is obtained for the
plurality function if $\rho \geq  0$. The case where $\rho < 0$
corresponds to the situation where the voting machine's rerandomization
mechanism favors votes that differ from the original vote. In this
case the theorem states that plurality will have the least correlation
between the intended outcome $f(X_1,\ldots,X_n)$ and the registered
outcome $f(Y_1,\ldots,Y_n)$. In the next subsection we discuss
applications of the result for hardness of approximation.

\subsubsection{Hardness of approximating MAX-q-CUT}

For NP-hard optimization problems in theoretical computer science research is conducted to find
polynomial time approximation algorithms that are guaranteed to find a
solution with value within a certain constant of the optimal value.
Hardness of approximation results on the other hand bound the
achievable approximation constants away from $1$.
For some problems, tight hardness results have been show where the
bound matches the best known polynomial time approximation
algorithm. For instance, H{\aa}stad \cite{Hastad:97} showed that for
MAX-E3-SAT one cannot improved upon the simple randomized algorithm
picking assignments at random thus achieving an approximation ratio of
$\frac{7}{8}$.

In general, for constraint satisfaction problems (CSP's) where the
object is
to maximize the number of satisfied predicates selected from a set of
allowed predicates and applied to a given set of variables,
algorithms based on relaxations to semi-definite programming (SDP),
first introduced by Goemans and Williamson~\cite{GoemansWilliamson:95}
has proved very successful.

Still optimal hardness results are not known for many CSP's.
One promising direction forward is the Unique Games Conjecture (UGC),
a strengthened form of the PCP Theorem introduced by Khot~\cite{Khot:02}.
Although the UGC remains open, hardness results for many
problems has since been proved under the assumption of the UGC, including
optimal results for
MAX-CUT~\cite{KKMO:07,MoOdOl:09} and
VERTEX-COVER~\cite{KhotRegev:03},
and improved results for
SPARSEST-CUT~\cite{Chawla:05, KhotVishnoi:05}.
Recently Raghavendra~\cite{Raghavendra:08} showed tight hardness results
for any MAX-CSP assuming the UGC, albeit without giving explicit optimal
approximation constants.

In Appendix \ref{sec:apxAlgo}
we consider one such problem that is known to be related to Plurality is Stablest.  
In the {\em MAX-q-CUT}  or {\em the Approximate q-Coloring problem}, we are  
given a graph (possibly edge weighted)
and we seek a $q$-coloring of the vertices that
maximizes the number (or weight) of edges between differently colored vertices.

\begin{definition}
  The weighted MAX-q-CUT problem, $\calM_q(V,E,w)$, is defined on a graph
  $(V,E)$ with a weight function $w:E \rightarrow [0,1]$ assigning
  a weight to each edge.
  A q-cut $l: V \rightarrow [q]$
  is a partition of the vertices into q parts.
  The value of a q-cut $l$ is
  \begin{equation}
    \VAL_l(\calM_q) = \sum_{(u,v)\in E : l(u) \neq l(v)} w_{(u,v)}
  \end{equation}
  The value of $\calM_q$ is
  \begin{equation}
    \VAL(\calM_q) = \max_l \VAL_l(\calM_q)
  \end{equation}
\end{definition}

Frieze-Jerrum gave an explicit SDP relaxation of MAX-q-CUT (see Appendix
\ref{sec:apxAlgo}) which was rounded using the standard simplex
partition of Conjecture \ref{conj:2GaussQSplitIsop}.
In Appendix \ref{sec:apxAlgo} we show that Conjecture \ref{conj:2GaussQSplitIsop} implies that this
is optimal.
\begin{theorem}
  \label{thm:condMaxqCutResults}
  Assume Conjecture \ref{conj:2GaussQSplitIsop} and the UGC.
  Then, for any $\epsilon>0$ there exist a polynomial time algorithm
  that approximates MAX-q-CUT within $\alpha_q-\epsilon$ while
  it is NP-hard to approximate MAX-q-CUT within
  $\alpha_q+\epsilon$.
  \newline
  Here,
  \begin{equation}
    \label{eq:alphaqopt}
    \alpha_q =
    \inf_{-\frac{1}{q-1} \le \rho \le 1}
    \frac{q}{q-1}
    \frac{1-q I(\rho)}
    {1-\rho}
  \end{equation}
  where $q I(\rho)$ is the noise stability of a standard simplex
  partition $S_1, \ldots, S_q$ of $\Reals^{q-1}$, i.e.
  \begin{equation}
    q I(\rho) = \P((X,Y) \in S_1^2\cup \dots \cup S_q^2)
  \end{equation}
  where $X,Y \sim \Norm(0, I_{q-1})$ are jointly normal
  with $\Cov(X,Y)=\rho I_{q-1}$.
\end{theorem}

We note that $\alpha_2 \approx 0.878567$ is the Goemans-Williamson
constant ~\cite{GoemansWilliamson:95}.
It is conjectured that \eqref{eq:alphaqopt} attains it minimum at
$\rho=-\frac{1}{q-1}$ for any $k \ge 3$ (but not for $k=2$).
This was verified numerically in~\cite{KlerkPW:04} for $k=3\ldots 10$,
where $\alpha_3, \ldots, \alpha_{10}$ were also computed.
For instance, $\alpha_3 \approx 0.836008$ and $\alpha_4 \approx
0.857487$.

We further comment briefly on the results of~\cite{Raghavendra:08}.
Since MAX-q-CUT is an example of MAX-CSP with a
single predicate, ~\cite{Raghavendra:08} give an
optimal approximation algorithm for MAX-q-CUT and an algorithm for
computing the optimal approximation constant.
However, the complexity of both these
algorithms depends heavily on the precision $\epsilon$.
In fact, the running time is doubly exponential in $1/\eps$.
In contrast our results (assuming the SSC) gives the optimal
approximation constant as simple optimization problem in one variable.

\subsection{Support for the SSC}
To support the Standard Simplex Conjecture we first note that it is a
natural extension of Theorem~\ref{thm:Borell}.
Moreover, by Theorem~\ref{thm:plurStabAppl} and
\ref{thm:plurStabApplRev} it is
(for $\rho \in [-\frac{1}{q-1},1]$) equivalent to the Plurality is
Stablest conjecture which is a natural extension of the Majority is
Stablest theorem. By \eqref{eq:plurStabAsympt} this extension
holds asymptotically as $q \rightarrow \infty$.
In the limit as $\rho \rightarrow 1$ further support is given by the
Double Bubble Theorem in Gaussian space as we explain next.

\subsubsection{The Double Bubble Theorem}

The famous Double Bubble Theorem~\cite{Hutchings:02} determines the minimal
area that encloses and separates two fixed volumes in
$\Reals^3$. The optimal partition is given by two spheres which intersect at an $120
\deg$ angle having a separating membrane in the plane of the
intersection. The proof of this theorem is the culmination of a long line of work answering a conjecture which was open
for more than a century.
\begin{figure}[h]
  \centering
\psset{xunit=.5pt,yunit=.5pt,runit=.5pt}
\begin{pspicture}(110.99920654,65.99921417)
{
\newrgbcolor{curcolor}{0 0 0}
\pscustom[linewidth=0.99921262,linecolor=curcolor]
{
\newpath
\moveto(78.49961,65.49960817)
\curveto(69.21159,65.49960817)(60.84703,61.54534017)(54.99961,55.21835817)
\curveto(49.21647,60.94647017)(41.27623,64.49960817)(32.49961,64.49960817)
\curveto(14.83561,64.49960817)(0.49961,50.16360817)(0.49961,32.49960817)
\curveto(0.49961,14.83560817)(14.83561,0.49960817)(32.49961,0.49960817)
\curveto(41.7733,0.49960817)(50.1222,4.47092317)(55.96836,10.78085817)
\curveto(61.75273,5.04414217)(69.71552,1.49960817)(78.49961,1.49960817)
\curveto(96.16361,1.49960817)(110.49961,15.83560817)(110.49961,33.49960817)
\curveto(110.49961,51.16360817)(96.16361,65.49960817)(78.49961,65.49960817)
\closepath
}
}
{
\newrgbcolor{curcolor}{0 0 0}
\pscustom[linewidth=1,linecolor=curcolor]
{
\newpath
\moveto(54.78077,55.42381517)
\curveto(55.35199,10.56808617)(55.35585,10.26500717)(55.35585,10.26500717)
}
}
\end{pspicture}
  \caption{A double bubble in $\Reals^2$}
\end{figure}

An analogous question can be asked in Gaussian space, $\Reals^n$ equipped
with a standard Gaussian density and the techniques and results used in the proof of the
Double Bubble Theorem allow to find the partition of $\Reals^n (n \geq 2)$ into three volumes
each having Gaussian volume~$\frac{1}{3}$  minimizing the Gaussian surface area between the three volumes.
Indeed, the results of~\cite{Corneli:08} show that the optimal partition is the Peace Sign
partition, which can be
seen as the limit of the double bubble partition scaled up around one
point on the intersection.

This indicates that the partition in Conjecture
\ref{conj:2GaussQSplitIsop}
is optimal (at least for $q=3$ when $\rho \rightarrow 1$).
Indeed Conjecture~\ref{conj:2GaussQSplitIsop} is stronger than the results of~\cite{Corneli:08}.
It is easy to see that Conjecture~\ref{conj:2GaussQSplitIsop} with
$q=3$ imply that the Peace Sign Partition is optimal
by taking the limit $\rho \to 1$ (this is done similarly to the way in which Borell's result~\cite{Borell:85} implies the classical Gaussian isoperimetric result, see Ledoux's
Saint-Flour lecture notes~\cite{DoGrLe:96}).

\subsection{Organization}
In Section~\ref{sec:preliminaries} we introduce the notation we use and
various definitions and results from previous work, while also proving
some useful properties of Gaussian noise stability.
In Section~\ref{sec:invariance} we describe the invariance principle
which is used to relate certain questions about discrete noise stability to
questions about Gaussian noise stability.
Then, in Section~\ref{sec:EGT} we prove the Exchangeable Gaussian
Theorem
and in Section~\ref{sec:lowInflBounds} we prove a general noise
stability bound for discrete low-influence functions.
The following sections treats various applications.
The first two applications are based on the EGT.
In Section~\ref{sec:appEGT} we show that the majority
function maximizes the probability of having a unique best winner in
Condorcet voting and
that majority is best for cosmic coin flipping among low-influence functions.
The next two applications are based on the Standard Simplex Conjecture.
In Section~\ref{sec:appPlur} we show that the Plurality is
Stablest Conjecture follows from the SSC - and essentially is
equivalent to the SSC.
Based on the results of Section~\ref{sec:appPlur},
we also include a proof in Appendix~\ref{sec:maxqcut} of the
optimality of the Frieze-Jerrum SDP for approximating MAX-q-CUT
given the SSC (and assuming the Unique Games Conjecture).

\section{Preliminaries}
\label{sec:preliminaries}
In this section we introduce some notation and recall various definitions
and results from~\cite{MoOdOl:09,Mossel:08}.
Furthermore, we derive some useful properties of Gaussian noise
stability in section~\ref{subsec:gaussianNoise}.

\subsection{Conventions}
To make it more clear whether we are working with functions on discrete space
$f:[q]^n \rightarrow \Reals^k$ or functions on continuous Gaussian
space $g:\Reals^n \rightarrow \Reals^k$ we will usually use $f$ to
denote discrete functions and $g$ to denote continuous functions.

For a discrete function we will write $\E f$ for $\E f(\omega)$ where
$\omega$ is uniformly selected from $[q]^n$ and
$\|f\|_2^2 = \E \ip{f(\omega)}{f(\omega)}$
and similarly for a continuous function we will write
$\E g$ for $\E g(X)$ where $X \sim \Norm(0,I_n)$
and $\|g\|_2^2 = \E \ip{g(X)}{g(X)}$.
We also say that $g \in L^2$ if $\|g\|_2 < \infty$.
\subsection{Multilinear polynomials}

Consider a product probability space
$(\Omega,\mu)=(\prod_{i=1}^n \Omega_i, \prod_{i=1}^n \mu_i)$.
We will be interested in functions
$f:\prod_{i=1}^n\Omega_i \rightarrow \Reals$
on such spaces.
For simplicity, we will assume that each $\mu_i$ as full support, i.e.
$\mu_i(\omega_i)>0, \forall \omega_i \in \Omega_i$.
Then clearly, for each coordinate $i$
we can create a (possibly orthonormal) basis of the form
\begin{equation}
  \label{eqn:calXBasis}
  \calX_i = (X_{i,0} = 1, X_{i,1}, \dots, X_{i,|\Omega_i|-1})
\end{equation}
where $E[X_{i,j}]=0$ for $j\ge1$,
for the space of functions $\Omega_i \rightarrow \Reals$.

\begin{definition}
  We call a finite sequence of (orthonormal) real-valued random variables where the first variable is the constant $1$ and the other variables have zero mean
  an (orthonormal) ensemble.
\end{definition}

Thus,
$\calX=(\calX_1,\dots,\calX_n)$
is an independent sequence of (possibly orthonormal) ensembles.
We will only be concerned with independent sequences of ensembles,
however we will not always require the ensembles to be orthonormal
Another type of ensembles are the Gaussian ensembles,
of which an independent sequence is typically denoted by
$\calZ=(\calZ_1, \dots, \calZ_n)$ where
$\calZ_i = (Z_{i,0} = 1, Z_{i,1}, \dots, Z_{i,m_i})$
and each $Z_{i,j}$ is a standard Gaussian variable.

\begin{definition}
  A multi-index $\sigma$ is a sequence of numbers
  $(\sigma_1, \dots, \sigma_n)$ such that $\sigma_i \ge 0, \forall i$.
  The degree $|\sigma|$ of $\sigma$ is $\left|\{i \in [n]:\sigma_i>0\}\right|$.
  Given a set of indeterminates $\{x_{i,j}\}_{i\in[n], 0 \le j \le m_i}$,
  let $x_\sigma=\prod_{i=1}^n x_{i,\sigma_i}$.
  A multilinear polynomial over such a set of indeterminates is an expression
  $Q(x)=\sum_\sigma c_\sigma x_\sigma$
  where $c_\sigma \in \Reals$ are constants.
\end{definition}

Continuing from \eqref{eqn:calXBasis} and
letting $X_\sigma = \prod_{i=1}^n X_{i,\sigma_i}$ it should be clear that
$\{X_\sigma\}$  forms a basis for functions
$\prod_{i=1}^n\Omega_i \rightarrow \Reals$, hence any function
$f:\prod_{i=1}^n\Omega_i \rightarrow \Reals$
can be expressed as a multilinear polynomial $Q$ over $\calX$:
\begin{equation}
  \label{eq:fQ}
  f(\omega_1,\dots,\omega_n)
  = Q(\calX_1, \dots, \calX_n)
  = \sum_\sigma c_\sigma X_\sigma
\end{equation}

\begin{definition}
  The degree of a multilinear polynomial $Q$ is
  \begin{equation}
    \deg Q = \max_{\sigma:c_\sigma \neq 0} |\sigma|
  \end{equation}
  We will also use the notation $Q^{\le d}$ to denote the truncated multilinear polynomial
  \begin{equation}
    Q^{\le d}(x)=\sum_{\sigma:|\sigma|\le d} c_\sigma x_\sigma
  \end{equation}
  and the analogous for $Q^{=d}$ and $Q^{> d}$.
\end{definition}

\begin{definition}
  Given a multilinear polynomial $Q$ over an independent sequence of ensembles
  $\calX=(\calX_1, \dots, \calX_n)$,
  the influence of the $i$:th coordinate on $Q(\calX)$ is
  \begin{equation}
    \Inf_i Q(\calX)
    =
    \E\left[
      \Var[Q(\calX)|\calX_1, \dots, \calX_{i-1}, \calX_{i+1}, \dots
      \calX_{n}]
    \right]
  \end{equation}
  We also define the $d$-degree influence of the $i$:th coordinate as
  \begin{equation}
    \Inf_i^{\le d} Q(\calX)
    =
    \Inf_i Q^{\le d}(\calX)
  \end{equation}
\end{definition}

Note that neither the degree nor influences of $Q(\calX)$ depends on
the actual basis selected in \eqref{eqn:calXBasis}, hence we can write
$\deg f = \deg Q$,
$\Inf_i f = \Inf_iQ(\calX)$
and
$\Infd_i f = \Inf_iQ^{\le d}(\calX)$.

\subsection{Bonami-Beckner noise}

Let us first define the Bonami-Beckner noise operator.
\begin{definition}
  \label{def:Tp}
  Let
  $(\Omega,\mu)=(\prod_{i=1}^n \Omega_i, \prod_{i=1}^n \mu_i)$.
  be a finite product probability space
  and $\alpha$ the minimum probability of any atom in any $\Omega_i$.
  For $-\frac{\alpha}{1-\alpha} \le \rho \le 1$
  the Bonami-Beckner operator on functions
  $f:\prod_{i=1}^n\Omega_i \rightarrow \Reals^k$ is defined by
  \begin{equation}
    T_\rho f(\omega_1, \dots, \omega_n) =
    \E[f(\lambda_1, \dots, \lambda_n)| \omega_1, \dots \omega_n]
  \end{equation}
  where each $\lambda_i$ is independently selected from the conditional distribution
  \begin{equation}
    \mu_i(\lambda_i|\omega_i)=\rho 1_{\{\lambda_i=\omega_i\}} + (1-\rho)\mu_i(\lambda_i)
  \end{equation}
\end{definition}

For $\rho \in [0,1]$ this is equivalent to $T_\rho f$ being the expected value of $f$ when each coordinate independently is rerandomized with probability $1-\rho$.

\subsection{Orthonormal ensembles}
Most of the time we will work with \emph{orthonormal} ensembles.
Using independence and linearity of expectation it is easy to see that
if $Q(\calX) = \sum_\sigma c_\sigma \calX_\sigma$ is a multilinear polynomial over an independent sequence of \emph{orthonormal} ensembles, then

\begin{eqnarray}
  \label{eq:fourierExpressions}
  \begin{aligned}
    &\E[Q(\calX)] = c_{\vec{0}} \text{ ; } &\Var[Q(\calX)] =
    \sum_{\sigma : |\sigma|>0} c_\sigma^2 \text{ ; } &\Inf_i Q(\calX) =
    \sum_{\sigma : \sigma_i>0} c_\sigma^2
    \\
    &\E[Q(\calX)^2] = \sum_{\sigma} c_\sigma^2 \text{ ; }& T_\rho
    Q(\calX) = \sum_\sigma \rho^{|\sigma|} c_\sigma X_\sigma \text{
      ; }
    & \Infd_i Q(\calX) = \sum_{\sigma : \left\{\substack{\sigma_i>0 \\ |\sigma| \le d}\right.} c_\sigma^2
  \end{aligned}
\end{eqnarray}

Combining these expressions it is also easy to see that
$\Infd_i f$ is convex in $f$ and satisfies the
following bound on the sum of low-degree influences:
\begin{equation}
  \label{eq:dInflBound}
  \sum_{i=1}^n \Infd_i f
  \le
  d \Var f
\end{equation}

\subsection{Vector-valued functions}
Since we will work extensively with vector-valued functions we make
the following definitions:
\begin{definition}
  For a vector-valued function $f=(f_1, \dots, f_k)$, let
  \begin{equation}
    \Var f
    =
    \sum_{j=1}^k \Var f_j
    \text{ , }\,\,
    \Inf_i f
    =
    \sum_{j=1}^k \Inf_i f_j
  \end{equation}
  and similarly for $\Infd_i$.
\end{definition}

Thus \eqref{eq:dInflBound} holds even for vector-valued $f$.
Also, all expressions in
\eqref{eq:fourierExpressions} hold for vector-valued multilinear
polynomials $Q(\calX) = \sum_\sigma c_\sigma \calX_\sigma$, where
$c_\sigma \in \Reals^k$ and $\calX$ is an independent sequence of
\emph{orthonormal} ensembles, if we replace $c_\sigma^2$ with
$\|c_\sigma\|_2^2$ and $\E[Q(\calX)^2]$ by
$\|Q(\calX)\|^2_2$.

Finally, by expressing functions
$f:[q]^n \rightarrow \Reals^k$
under the uniform measure on the input space $[q]^n$ as a multilinear
polynomial
\begin{equation}
  f(\omega) = \sum_\sigma c_\sigma \prod_{i=1}^n X_{i,\sigma_i} (\omega_i)
\end{equation}
this lets us express the noise stability of Definition \ref{def:SpQ} as
\begin{equation}
  \label{eq:SpFourier}
  \Sp(f)
  =
  \E[\ip{f}{T_\rho f}]
  =
  \sum_{\sigma} \rho^{|\sigma|} \|c_\sigma\|_2^2
\end{equation}

\subsection{Correlated probability spaces}
It will be important for us to bound the effect of the Bonami-Beckner noise operator on functions on correlated probability spaces.
\begin{definition}
  Let $(\Omega_1 \times \Omega_2, \mu)$ be a correlated probability space. The correlation between $\Omega_1$ and $\Omega_2$ with respect to $\mu$ is then
  \begin{equation}
    \rho(\Omega_1, \Omega_2; \mu) = \sup_{f_i:\Omega_i \rightarrow \Reals, \Var f_i=1} \Cov(f_1(\omega_1), f_2(\omega_2))
  \end{equation}
  For $(\Omega_1 \times \dots \times \Omega_k, \mu)$ we let
  \begin{equation}
    \rho(\Omega_1, \dots, \Omega_k; \mu) = \max_{1\le i \le k} \rho
    \left(
      \Omega_i, \prod_{j \neq i} \Omega_j; \mu
    \right)
  \end{equation}
\end{definition}

The following lemma shows that the expected value of products of functions where corresponding coordinates form correlated probability spaces does not change by much when some small noise is applied to each coordinate:

\begin{lemma}\cite[Lemma 6.2]{Mossel:08}
  \label{lem:smoothing}
  Let $(\prod_{i=1}^n \Omega_i, \prod_{i=1}^n \mu_i)$
  be a finite product probability space where
  $\Omega_i=(\Omega_i^1, \dots, \Omega_i^k)$
  are correlated probability spaces with
  $\rho(\Omega_i^1, \dots, \Omega_i^k; \mu_i) \le \rho < 1$. 
  For $j=1\ldots k$,
  let $\calX^j = (\calX_1^j, \dots, \calX_n^j)$
  be an independent sequence of orthonormal ensembles such that $\calX_i^j$
  forms a basis for functions $\Omega_i^j \rightarrow \Reals$
  and
  $Q_1, \dots, Q_k$ multilinear polynomials bounded by
  $|Q_j(\calX^j)| \le 1$.
  Then, for all $\epsilon>0$ there exists a $\gamma=\gamma(\epsilon,\rho)>0$ such that
  \begin{equation}
    \left|
      \E \prod_{j=1}^k Q_j(\calX^j)
      -
      \E \prod_{j=1}^k T_{1-\gamma} Q_j(\calX^j)
    \right|
    \le \epsilon \cdot k
  \end{equation}
\end{lemma}

To verify the assumption
$\rho(\Omega_i^1, \dots, \Omega_i^k; \mu_i) < 1$ the following lemma
is useful:

\begin{lemma}\cite[Lemma 2.9]{Mossel:08}
  \label{lem:connRho}
  Let $(\Omega_1 \times \Omega_2, \mu)$ be a correlated probability
  space such that $\mu(\omega_1,\omega_2) \ge \alpha$ or
  $\mu(\omega_1,\omega_2)=0$
  for all $\omega_1,\omega_2$.
  Define a bipartite graph
  $G=(\Omega_1 \cup \Omega_2, E)$ where
  $(a,b) \in E$ if $\mu(a,b) > 0$. Then, if $G$ is connected, then
  \begin{equation}
    \rho(\Omega_1,\Omega_2; \mu) \le 1 -\alpha^2/2
  \end{equation}
\end{lemma}

\subsection{Gaussian noise stability}
\label{subsec:gaussianNoise}

\begin{definition}
  For $\rho \in [-1,1]$,
  the Ornstein-Uhlenbeck operator $U_\rho$ is defined on functions
  $g:\Reals^n \rightarrow \Reals^k$ such that
  $g \in L^2$,
  where $X \sim \Norm(0, I_n)$,
  by
  \begin{equation}
    U_\rho g(x) = \E \left[ g(\rho x + \sqrt{1-\rho^2} \xi) \right]
  \end{equation}
  where $\xi \sim \Norm(0, I_n)$.
\end{definition}

It is easy to see that if $\calZ = (\calZ_1, \dots, \calZ_n)$ is a Gaussian sequence of independent ensembles and $Q(\calZ) = \sum_\sigma c_\sigma \calZ_\sigma$, then
\begin{equation}
  U_\rho Q(\calZ) = \sum_\sigma \rho^{|\sigma|} c_\sigma \calZ_\sigma
\end{equation}
Thus $U_\rho$ and $T_\rho$ acts identically on multi-linear polynomials over Gaussian sequences of independent ensembles.

Analogous to the expression
\eqref{eq:SpFourier}
of discrete noise stability in terms of the Bonami-Beckner operator,
we define the Gaussian noise stability in terms of the
Ornstein-Uhlenbeck operator,
\begin{definition}
  For any $g:\Reals^n \rightarrow \Delta_q$, let
  \begin{equation}
    \Sp(g) = \E[\ip{g}{U_\rho g}]
  \end{equation}
\end{definition}
\noindent
Note that we use the same notation $\Sp$ for both discrete and
Gaussian noise stability. The intended kind of noise should always be
clear from the context.

A convenient property of the Ornstein-Uhlenbeck operator is that it
creates continuous functions. The following result is well known:
\begin{lemma}
  \label{lem:OrnsteinSmooth}
  For any $\rho \in (-1,1)$ and
  $g:\Reals^n \rightarrow \Delta_q$,
  $U_\rho g (x)$ is continuous in $x$.
\end{lemma}

\begin{proof}
  \begin{equation*}
    \| U_\rho g (x) - U_\rho g (y) \|_2
    =
    \| \E g(U) -  \E g(V) \|_2
  \end{equation*}
  where
  $U=\rho x + \sqrt{1-\rho^2} \xi_x$
  ,
  $V=\rho y + \sqrt{1-\rho^2} \xi_y$
  and $\xi_x, \xi_y \sim \Norm(0,I_n)$.

  First note that if $X \in \Norm(\mu,1)$ and $Y \in \Norm(-\mu,1)$,
  then the total variation distance between $X$ and $Y$ is
  \begin{align*}
    d_{TV}(X,Y)
    & =
    \frac{1}{2\sqrt{2\pi}} \int_{\Reals}
    \left|
      e^{-\frac{(x+\mu)^2}{2}}
      -
      e^{-\frac{(x-\mu)^2}{2}}
    \right|
    dx
    \\
    & =
    \frac{1}{2\sqrt{2\pi}} \int_{\Reals}
    e^{-\frac{x^2}{2}-\frac{\mu^2}{2}}
    \left|
      e^{-x \mu}
      -
      e^{x \mu}
    \right|
    dx
    \le
    \frac{1}{2\sqrt{2\pi}}
    \int_{\Reals}
    e^{-\frac{x^2}{2}}
    2 e^{x \mu}
    dx
    \\
    & \le
    \E_{Z \in \Norm(0,1)} [e^{Z \mu}]
    =
    e^{\frac{\mu^2}{2}}
  \end{align*}
  Hence,
  \begin{equation*}
    d_{TV}(U_i,V_i)
    =
    d_{TV}
    \left(
      \frac{U_i-\rho\frac{x_i+y_i}{2}}{\sqrt{1-\rho^2}}
      ,
      \frac{V_i-\rho\frac{x_i+y_i}{2}}{\sqrt{1-\rho^2}}
    \right)
    \le
    e^{\rho^2\frac{(x_i-y_i)^2}{8}}
  \end{equation*}
  and
  \begin{equation*}
    d_{TV}(U,V)
    \le
    \sum_{i=1}^n
    e^{\rho^2\frac{(x_i-y_i)^2}{8}}
    \rightarrow 0
    \text{ as }
    \|x-y\|_2 \rightarrow 0
  \end{equation*}
  Since we can couple $U$ and $V$ such that they are equal except
  with probability $d_{TV}(U,V)$, we have
  $
  \| \E g(U) - \E g(V) \|_2
    \rightarrow 0
    \text{ as }
    \|x-y\|_2 \rightarrow 0
  $
  as needed.
\end{proof}

Also, applying some small noise will not affect the noise stability
much,
\begin{lemma}
  \label{lem:OrnsteinConvergence}
  For $\rho \in [-1,1]$
  and
  $g:\Reals^n \rightarrow \Delta_q$,
  \begin{equation}
    \Sp(U_{1-\epsilon} g) \rightarrow \Sp(g)
    \text{ as }
    \epsilon \rightarrow 0
  \end{equation}
\end{lemma}
\begin{proof}
  \begin{eqnarray*}
    \left|
      \Sp(U_{1-\eps} g)
      -
      \Sp(g)
    \right|
    &=&
    \left|
      \E\ip{U_{1-\eps} g}{U_\rho U_{1-\eps} g}
      -
      \E\ip{g}{U_\rho g}
    \right|
    =
    \\
    &=&
    \left|
      \E\ip{U_{1-\eps} g - g}{U_\rho U_{1-\eps} g}
      +
      \E\ip{g}{U_\rho U_{1-\eps} - U_\rho g}
    \right|
    \le
    \\
    &\le&
    \|U_{1-\eps} g - g\|_2
    +
    \|U_{1-\eps}U_\rho - U_\rho g\|_2
  \end{eqnarray*}
  where the inequality follows from Cauchy-Schwarz and commutativity
  of $U_\rho$ and $U_{1-\eps}$.
  By \cite[Theorem 4.20]{Janson:97}, if $g:\Reals^n \rightarrow
  \Reals$, $g \in L^2$, then
  \begin{equation}
    U_{1-\eps} g \rightarrow g \text{ in } L^2 \text{ as } \eps \rightarrow 0
  \end{equation}
  Clearly, this extends to vector-valued functions as well,
  hence the result follows.
\end{proof}

Analogous to the discrete setting we say that
$g:\Reals^n \rightarrow \Delta_q$
is balanced if
$\E[g(X)]=\frac 1 q \vec{1}$ for $X \sim \Norm(0,1)$.

The following lemma shows for any fuzzy partition a non-fuzzy
partition with almost the same expectation and noise stability (as
measured in Theorem \ref{thm:nGauss2SplitIsop} and Conjecture
\ref{conj:2GaussQSplitIsop}) can be created.
\begin{lemma}
  \label{lem:nonFuzziness}
  Fix $\rho \in \left[-\frac{1}{k-1},1\right]$ and $q_0 \le q$.
  Suppose $X_1,\dots,X_k \sim \Norm(0, I_n)$ and
  $\Cov(X_i,X_j)=\rho I_n$ for $i\neq j$.
  Then, for any $\epsilon>0$ and
  $g_1, \ldots, g_k:\Reals^n \rightarrow \Delta_q$,
  there exist $h_1, \ldots, h_k:\Reals^n \rightarrow E_q$
  such that
  \begin{equation}
    \label{eq:fuzzyDev1}
    \sum_{i=1}^q
    \left| \E h_{j,i} - \E g_{j,i} \right| \le q \epsilon
    \,\,,
    \forall j
  \end{equation}
  and
  \begin{equation}
    \label{eq:fuzzyDev2}
    \left|
      \E \sum_{i=1}^{q_0} \prod_{j=1}^k h_{j,i}(X_j)
      -
      \E \sum_{i=1}^{q_0} \prod_{j=1}^k g_{j,i}(X_j)
    \right|
    \le \epsilon
  \end{equation}
\end{lemma}
\noindent

\begin{proof}
  Assume first that $\rho \in \left(-\frac{1}{k-1},1\right)$ so that
  the normal distribution is non-degenerate.
  Discretize $\Reals^n$ with cubes $[0,\delta)^n$, i.e.
  write $\Reals^n = \delta \mathbb{Z}^n \times [0,\delta)^n$.
  where $\delta \mathbb{Z}^n$ denotes the n-dimensional integer
  lattice scaled by a factor $\delta$.

  Let $Z_{i,j} = \delta \left\lfloor \frac{X_{i,j}}{\delta} \right\rfloor$
  so that $Z_i$ denotes the cube $X_i$ is in,
  and let $U_{i,j}$ be i.i.d. uniform on $[0,\delta]$,
  independent of $X_1, \ldots X_k$.

  Further let $\eta$ be the density of $(X_1, \ldots, X_k)$ and
  $\tilde \eta$ the density of $(Z_1+U_1, \ldots, Z_k + U_k)$.
  By continuity of $\eta$ we have pointwise convergence,
  \begin{equation}
    \tilde{\eta} (x)
    \rightarrow
    \eta (x)
    \text{ as }
    \delta \rightarrow 0
  \end{equation}
  By dominated convergence, this implies that we can choose $\delta$ so that
  \begin{equation}
    \int_{\Reals^{nk}}
    \left|
      \eta(x)
      -
      \tilde\eta(x)
    \right|
    dx
    \le
    \frac{\epsilon}{2}
  \end{equation}
  Hence, for any $f:\Reals^{nk} \rightarrow [0,1]$, we have
  \begin{equation}
    \label{eq:scheffe}
    \left|
      \int_{\Reals^{nk}} f(x) \eta(x) dx
      -
      \int_{\Reals^{nk}} f(x) \tilde\eta(x) dx
    \right|
    \le
    \int_{\Reals^{nk}}
    f(x)
    \left|
      \eta(x)
      -
      \tilde\eta(x)
    \right|
    dx
    \le \frac{\epsilon}{2}
  \end{equation}

  Each non-fuzzy function $h_j$ is constructed from $g_j$ by transferring masses internally in each cube.
  More specifically,
  $h_j$ is defined arbitrarily on each cube with the only restriction that
  \begin{equation}
    \label{eq:sameE}
    \E[h_j(Z_1 + U_1) | Z_1]
    =
    \E[g_j(Z_1 + U_1) | Z_1]
  \end{equation}
  (For instance, if $\E[h(Z_1 + U_1) | Z_1=z_1] = \mu$,
  then we may divide the cube
  $z_1 + [0,\delta)^n$ into $q$ parts of conditional measure $\mu_1, \ldots \mu_q$
  and assign the value $e_1, \ldots, e_q$ respectively to each part.)
  Thus,
  \begin{eqnarray*}
    \nonumber
      \E \sum_{i=1}^{q_0} \prod_{j=1}^k g_{j,i}(Z_j+U_j)
      =
      \E \sum_{i=1}^{q_0} \prod_{j=1}^k \E[g_{j,i}(Z_j+U_j)|Z_j]
      =
      \\
      \label{eq:comb1}
      =
      \E \sum_{i=1}^{q_0} \prod_{j=1}^k \E[h_{j,i}(Z_j+U_j)|Z_j]
      =
      \E \sum_{i=1}^{q_0} \prod_{j=1}^k h_{j,i}(Z_j+U_j)
  \end{eqnarray*}
  Applying \eqref{eq:scheffe} twice gives \eqref{eq:fuzzyDev2}.
  Similarly
  \begin{equation*}
      \E g_{j,i}(Z_1+U_1)
      =
      \E [\E[g_{j,i}(Z_1+U_1)|Z_1]]
      =
      \E [\E[h_{j,i}(Z_1+U_1)|Z_1] ]
      =
      \E h_{j,i}(Z_1+U_1)
  \end{equation*}
  and two more applications of \eqref{eq:scheffe} gives
  $|\E g_{j,i}(X_1) - \E h_{j,i}(X_1)| \le \epsilon$
  and \eqref{eq:fuzzyDev1} follows.

  The two degenerate cases can be handled in a similar way by using a
  density with respect to a lower dimensional Lebesgue measure.
\end{proof}

We also need a simple result that states that, for instance, almost balanced functions cannot be much more stable than balanced functions:
\begin{lemma}
  \label{lem:almostBalanced}
  Fix $\rho \in [-\frac{1}{k-1},1]$, $q_0 \le q$ and
  $\mu_1, \ldots, \mu_k \in \Reals^q$.
  Suppose $X_1,\dots,X_k \sim \Norm(0, I_n)$ are jointly normal with
  $\Cov(X_i,X_j)=\rho I_n$ for $i\neq j$.
  Let $g_1, \ldots, g_k:\Reals^n \rightarrow E_q$ with
  \begin{equation}
    \sum_{i=1}^q
    \left| \E g_{j,i} - \mu_j \right|
    = \delta_j
  \end{equation}
  Then, there exist $h_1, \ldots, h_k:\Reals^n \rightarrow E_q$
  with $\E h_j = \mu_j$
  such that
  \begin{equation}
    \left|
      \E \sum_{i=1}^{q_0} \prod_{j=1}^k g_{j,i}(X_j)
      -
      \E \sum_{i=1}^{q_0} \prod_{j=1}^k h_{j,i}(X_j)
    \right|
    \le
    \sum_{j=1}^k \frac{\delta_j}{2}
  \end{equation}
\end{lemma}

\begin{proof}
  Clearly, it is enough to change the value of $g_j$ on a set of
  Gaussian measure $\frac{\delta_j}{2}$ (and such sets can easily be
  find since the Gaussian density is continuous).
  Thus, we can create a function $h_j$ with $\E h_j = \mu_j$ such that
  $P(g_j(X_j) \neq h_j(X_j)) = \frac{\delta_j}{2}$, and the
  result follows by the union bound.
\end{proof}

\section{An invariance principle}
\label{sec:invariance}
Let $f:\prod_{i=1}^n\Omega_i \rightarrow \Reals$ be a function on a finite product probability space and express it as a multilinear polynomial $Q(\calX)$ over an independent sequence of orthonormal ensembles as in \eqref{eq:fQ}.
The invariance principle of~\cite{MoOdOl:09} (see also earlier results
in ~\cite{Rotar:75}), shows that if $Q$ has low degree and each coordinate has small influence then the distribution of $Q(\calX)$ does not change by much if we replace the variables $X_{i,j}$ with independent standard Gaussians $Z_{i,j}$.

In~\cite{Mossel:08} the invariance principle was extended to the case of vector-valued functions $f=(f_1,\dots,f_k)$ where $f_j:\prod_{i=1}^n\Omega_i \rightarrow \Reals$ for each j.

\begin{theorem}(\cite{Mossel:08}, Theorem 4.1 and 3.16)
  \label{thm:invariancePrinciple}
  Let $(\prod_{i=1}^n \Omega_i, \prod_{i=1}^n \mu_i)$
  be a finite product probability space,
  $\alpha>0$ the minimum probability of any atom in any $\mu_i$
  and $\calX=(\calX_1, \dots, \calX_n)$ an independent sequence of orthonormal ensembles such that $\calX_i$ is a basis for functions $\Omega_i \rightarrow \Reals$.
  Let $Q$ be a k-dimensional multilinear polynomial such that
  $\Var Q_j(\calX) \le 1$, $\deg Q_j \le d$ and
  $\Inf_i Q_j(\calX) \le \tau$.
  Finally, let $\Psi:\Reals^k \rightarrow \Reals$ be a $\mathcal{C}^3$ function
  with $|\Psi^{(\vec r)}| \le B$ for $|\vec r| = 3$.
  Then,
  \begin{equation}
    \left|
      \E\Psi(Q(\calX)) - \E\Psi(Q(\calZ))
    \right|
    \le 2dBk^3\left(8/\sqrt{\alpha}\right)^d \sqrt{\tau}
    = \bigoh({\sqrt \tau})
  \end{equation}
  where $\calZ$ is an independent sequence of standard Gaussian ensembles.
\end{theorem}

As suggested in~\cite[Corollary 4.3]{Mossel:08},
since neither $\Var Q_j(\calX)$, $\deg Q_j$ or $\Inf_i Q_j$ depend on whether the ensembles are orthonormal, we can simply replace the orthonormal requirement by a matching covariance structure requirement.

\begin{definition}
  We say that two independent sequences of ensembles $\calX=(\calX_1,
  \dots, \calX_n)$ and $\calY=(\calY_1, \dots, \calY_n)$ have a
  matching covariance structure if for all $i$,
  $|\calX_i|=|\calY_i|$
  and
  $\E[\calX_i^t \calX_i] = \E[\calY_i^t \calY_i]$.
\end{definition}

\begin{theorem}
  \label{thm:invariancePrinciple2}
  Let $\calX=(\calX_1, \dots, \calX_n)$ be an independent sequence of
  ensembles,
  such that $\P(\calX_i = x) \ge \alpha > 0, \forall i,x$.
  Let $Q$ be a k-dimensional multilinear polynomial such that
  $\Var Q_j(\calX) \le 1$, $\deg Q_j \le d$ and
  $\Inf_i Q_j(\calX) \le \tau$.
  Finally, let $\Psi:\Reals^k \rightarrow \Reals$ be a $\mathcal{C}^3$ function
  with $|\Psi^{(\vec r)}| \le B$ for $|\vec r| = 3$.
  Then,
  \begin{equation}
    \left|
      \E\Psi(Q(\calX)) - \E\Psi(Q(\calZ))
    \right|
    \le 2dBk^3\left(8/\sqrt{\alpha}\right)^d \sqrt{\tau}
    = \bigoh({\sqrt \tau})
  \end{equation}
  where $\calZ$ is an independent sequence of Gaussian ensembles with the same covariance structure as $\calX$.
\end{theorem}

\begin{proof}
  For each $i$, let $\Omega_i$ be the $\sigma$-algebra generated by the variables in $\calX_i$. Since $\alpha>0$, $\Omega_i$ is finite, hence we can find an
  orthonormal ensemble $\calX_i'$ which is a basis for
  $\Omega_i \rightarrow \Reals$
  and a linear transformation $A_i$ such that
  $\calX_i = \calX_i' A_i$.
  Let $\calZ'$ be any standard Gaussian ensemble and $\calZ_i = \calZ_i' A_i$.
  Then $\calZ$ has the same covariance structure as $\calX$.
  Let $Q'$ be the multilinear polynomial defined by $Q'(\calX') =
  Q(\calX'_1 A_1, \dots, \calX'_n A_n)$.
  The result then follows by applying Theorem \ref{thm:invariancePrinciple} to
  $Q'(\calX')$
  while noting that it has the same
  variances, degrees and influences as $Q(\calX)$.
\end{proof}

For our applications
we will need a version of Theorem \ref{thm:invariancePrinciple2} for
functions $\Psi$ which are not $\mathcal{C}^3$ functions.
Instead we will assume that $\Psi$ is Lipschitz continuous
with Lipschitz constant $A$, i.e.
$|\Psi(x)-\Psi(y)| \le A \|x-y\|_2$.

\begin{theorem}
  \label{thm:invariancePrinciple3}
  Let $\calX=(\calX_1, \dots, \calX_n)$ be an independent sequence of
  ensembles,
  such that $\P(\calX_i = x) \ge \alpha > 0, \forall i,x$.
  Let $Q$ be a k-dimensional multilinear polynomial such that
  $\Var Q_j(\calX) \le 1$, $\deg Q_j \le d$ and
  $\Inf_i Q_j(\calX) \le \tau$.
  Finally, let $\Psi:\Reals^k \rightarrow \Reals$ be
  Lipschitz continuous with Lipschitz constant $A$.
  Then,
  \begin{equation}
    \left|
      \E\Psi(Q(\calX)) - \E\Psi(Q(\calZ))
    \right|
    \le D_k A
    \left(
      d \left(8/\sqrt{\alpha}\right)^d \sqrt{\tau}
    \right)^{1/3}
    = \bigoh(\tau^{1/6})
  \end{equation}
  where $\calZ$ is an independent sequence of Gaussian ensembles with
  the same covariance structure as $\calX$
  and $D_k$ are universal constants.
\end{theorem}

To prove Theorem \ref{thm:invariancePrinciple3} we need the following
lemma which assures that Lipschitz continuous functions can be
approximated well by $\mathbb{C}^3$ functions.

\begin{lemma}
  \label{lem:c3Approx}
  Suppose $\Psi: \Reals^k \rightarrow \Reals$ is Lipschitz continuous, i.e.
  $|\Psi(x)-\Psi(y)| \le A \|x-y\|_2$ for some constant $A>0$.
  Then, for all $\lambda>0$ there exists a $\mathcal{C}^\infty$ function
  $\Psi_\lambda:\Reals^k \rightarrow \Reals$ such that
  $\forall x\in \Reals^k$ and $\forall \vec r : |\vec r| = r \ge 1$,
  \begin{enumerate}[i)]
  \item
    $|\Psi(x) - \Psi_\lambda(x)| \le A \lambda$
  \item
    $|\Psi_\lambda^{(\vec{r})}(x)| \le \frac{A B_{r,k}}{\lambda^{r-1}}$
  \end{enumerate}
  where $B_{r,k}$ are universal constants.
\end{lemma}
\begin{proof}
  Let $\mu$ denote the Lebesgue measure on $\Reals^k$ and
  let $\phi:\Reals^k \rightarrow \Reals$ be the k-dimensional bump function defined by
  \begin{equation}
    \phi(x) = \left\{
      \begin{array}{ll}
        C e^{-\frac{1}{1-\|x\|_2^2}} & \mbox{if } \|x\|_2 < 1 \\
        0 & \mbox{else}
      \end{array}
    \right.
  \end{equation}
  where the constant $C$ is chosen so that
  $\int_{x \in \Reals^k} \phi(x) \mu(dx)=1$.
  It is well-known that $\phi(x)$ is $\mathcal{C}^\infty$ with bounded
  derivatives, hence there exist constants $B_r < \infty$ such that
  $\sup_x |\phi^{(\vec r)}(x)| \le B_r$. 

  For $\lambda>0$, let
  $\phi_\lambda(x) = \frac{1}{\lambda^k}\phi(\frac x \lambda)$.
  Then $\int_{\|x\|_2 \le \lambda} \phi_\lambda(x) \mu(dx)=1$ and
  $|\phi_\lambda^{(\vec r)}(x)| \le \frac{B_r}{\lambda^{k+r}}$.
  Let $\Psi_\lambda = \Psi * \phi_\lambda$, i.e.
  \begin{equation}
    \Psi_\lambda(x) = \int_{\|x-t\|_2 \le \lambda} \phi_\lambda(x-t) \Psi(t) \mu(dt)
  \end{equation}
  By the mean value theorem,
  $\Psi_\lambda(x) = \Psi(t)$, for some $t: \|x-t\|_2 \le \lambda$.
  But $|\Psi(t)-\Psi(x)| \le A \|x-t\|_2 \le A \lambda$, which proves i).
  \newline
  Without loss of generality we may assume that
  $\vec r = \vec e_1 + \vec r_2$, where $\vec e_1=(1,0,\dots,0)^t$ is
  the first unit vector.
  Since $\Psi$ is bounded on $\|x-t\|_2 \le \lambda$,
  $\Psi_\lambda$ is $\mathcal{C}^\infty$ and for any $\vec{s}$,
  \begin{equation}
    \Psi_\lambda^{(\vec{s})}(x)
    =
    \int_{\|x-t\|_2 \le \lambda} \phi_\lambda^{(\vec{s})}(x-t) \Psi(t) \mu(dt)
  \end{equation}
  Thus we may write
  \begin{eqnarray*}
    \left|
      \Psi_\lambda^{(\vec{r})}(x)
    \right|
    &
    =
    &
    \left|
      \frac{\partial}{\partial x_1}
      \int_{\|x-t\|_2 \le \lambda} \phi_\lambda^{(\vec{r_2})}(x-t)
      \Psi(t) \mu(dt)
    \right|
    =
    \\
    &
    =
    &
    \left|
      \frac{\partial}{\partial x_1}
      \int_{\|t\|_2 \le \lambda} \phi_\lambda^{(\vec{r_2})}(t)
      \Psi(x-t) \mu(dt)
    \right|
    =
    \\
    &
    =
    &
    \left|
      \lim_{h \rightarrow 0}
      \int_{\|t\|_2 \le \lambda} \phi_\lambda^{(\vec{r_2})}(t)
      \frac{(\Psi(x+h \vec e_1 -t) - \Psi(x-t))}{h} \mu(dt)
    \right|
    =
    \\
    &
    =
    &
    \lim_{h \rightarrow 0}
    \left|
      \int_{\|t\|_2 \le \lambda} \phi_\lambda^{(\vec{r_2})}(t)
      \frac{(\Psi(x+h \vec e_1 -t) - \Psi(x-t))}{h} \mu(dt)
    \right|
    \le
    \\
    & \le &
    \lim_{h \rightarrow 0}
    \int_{\|t\|_2 \le \lambda} \left|\phi_\lambda^{(\vec{r_2})}(t)\right|
    \left|\frac{(\Psi(x+h \vec e_1 -t) - \Psi(x-t))}{h}\right| \mu(dt)
    \le
    \\
    & \le &
    \frac{B_{r-1}}{\lambda^{k+r-1}} A (2\lambda)^k
    =
    \frac{B_{r-1}}{\lambda^{r-1}} A 2^k
  \end{eqnarray*}
  Taking $B_{r,k} = B_{r-1} 2^k$ proves ii).
\end{proof}

\begin{proof}[Proof of Theorem \ref{thm:invariancePrinciple3}]
  Let $\Psi_\lambda$ be the approximation given by Lemma \ref{lem:c3Approx}.
  Then,
  \begin{eqnarray*}
    \left|
      \E\Psi(Q(\calX))
      -
      \E\Psi(Q(\calZ))
    \right|
    \le
    \left|
      \E\Psi_\lambda(Q(\calX))
      -
      \E\Psi_\lambda(Q(\calZ))
    \right|
    + 2 A \lambda
    \le
    \\
    \le
    \frac{2 A \epsilon}{\lambda^2}
    +
    2 A \lambda
    \mbox{ , where }
    \epsilon=d B_{3,k} k^3\left(8/\sqrt{\alpha}\right)^d \sqrt{\tau}
  \end{eqnarray*}
  where we have used Theorem \ref{thm:invariancePrinciple2}.
  Picking $\lambda=\epsilon^{1/3}$ and letting
  $D_k = 4 k B_{3,k}^{1/3}$ gives the result.
\end{proof}

Our final version of the invariance principle replaces the bounded
degree requirement with a smoothness requirement which can be achieved
by applying the Bonami-Beckner operator $T_{1-\gamma}$ on $Q(\calX)$
for some small $\gamma>0$. Later we will use Lemma \ref{lem:smoothing}
to show that this smoothing is essentially harmless for our applications.

\begin{theorem}
  \label{thm:invariancePrincipleFinal}
  Let $\calX=(\calX_1, \dots, \calX_n)$ be an independent sequence of
  ensembles,
  such that $\P(\calX_i = x) \ge \alpha > 0, \forall i,x$.
  Fix $\gamma, \tau \in (0,1)$ and
  let $Q$ be a k-dimensional multilinear polynomial such that
  $\Var Q_j(\calX) \le 1$,
  $\Var Q_j^{>d} \le (1-\gamma)^{2d}$
  and
  $\Inf_i Q_j^{\le d}(\calX) \le \tau$, where
  $d=\frac{1}{18}\log\frac{1}{\tau}/\log\frac{1}{\alpha}$.
  Finally, let $\Psi:\Reals^k \rightarrow \Reals$ be
  Lipschitz continuous with Lipschitz constant $A$.
  Then,
  \begin{equation}
    \left|
      \E\Psi(Q(\calX)) - \E\Psi(Q(\calZ))
    \right|
    \le C_k A
    \tau^{\frac{\gamma}{18} / \log \frac{1}{\alpha}}
  \end{equation}
  where $\calZ$ is an independent sequence of Gaussian ensembles with the same covariance structure as $\calX$
  and $C_k$ is a constant depending only on $k$.
\end{theorem}

To prove Theorem \ref{thm:invariancePrincipleFinal} we need following easy lemma which bounds the effect of small deviations on Lipschitz continuous functions.

\begin{lemma}
  \label{lem:psiDeltaNoise}
  Suppose $\Psi: \Reals^k \rightarrow \Reals$ is Lipschitz continuous,
  i.e.
  $|\Psi(x)-\Psi(y)| \le A \|x-y\|_2$ for some constant $A>0$.
  Then, for all random variables $X, \xi$ taking
  values in $\Reals^k$,
  \begin{equation}
    \left|
      \E \Psi(X+\xi)
      -
      \E \Psi(X)
    \right|
    \le
    A \left( \sum_{j=1}^k \E \xi_j^2 \right)^{1/2}
  \end{equation}
\end{lemma}
\begin{proof}
  \begin{eqnarray*}
    |\E \Psi(X+\xi) - \E \Psi(X)|
    \le
    \E |\Psi(X+\xi) - \Psi(X)|
    \le
    \E A \|\xi\|_2
    =
    \\
    =
    A \E \left(\sum_{j=1}^k \xi_j^2\right)^{1/2}
    \le
    A \left(\sum_{j=1}^k \E \xi_j^2\right)^{1/2}
  \end{eqnarray*}
\end{proof}

\begin{proof}[Proof of Theorem \ref{thm:invariancePrincipleFinal}]
  The proof is by truncation of $Q$ at degree
  $d=\frac{1}{18}\log\frac{1}{\tau}/\log\frac{1}{\alpha}$.
  Without loss of generality we may assume that $\alpha \le \frac 1 2$ (else, all random variables are constants and the result is trivial).
  Using Lemma \ref{lem:psiDeltaNoise} twice (with $\xi=Q^{>d}(\calX)$
  and $\xi=Q^{>d}(\calZ)$ respectively) and noting that Theorem
  \ref{thm:invariancePrinciple3} holds for all positive real values on $d$
  we find,
  \begin{eqnarray*}
    \left|
      \E\Psi(Q(\calX)) - \E\Psi(Q(\calZ))
    \right|
    \le
    \left|
      \E\Psi(Q^{\le d}(\calX)) - \E\Psi(Q^{\le d}(\calZ))
    \right|
    + 2A \sqrt{k} (1-\gamma)^d
    \\
    \le
    D_k A \left(16/\sqrt{\alpha}\right)^{d/3} \tau^{1/6}
    + 2A \sqrt{k} e^{-\gamma d}
  \end{eqnarray*}
  The result now follows by noting that
  \begin{equation*}
    e^{-\gamma d}
    =
    \tau^{\frac{\gamma}{18} / \log \frac{1}{\alpha}}
  \end{equation*}
  and
  \begin{eqnarray*}
    \left(16/\sqrt{\alpha}\right)^{d/3} \tau^{1/6}
    & = &
    e^{\frac{d}{6}\log{\frac{256}{\alpha}}} \tau^{1/6}
    =
    \tau^{-\frac{1}{6 \cdot 18} \log{\frac{256}{\alpha}}/
      \log{\frac{1}{\alpha}}} \tau^{1/6}
    \le
    \\
    & \le &
    \tau^{-\frac{1}{12}} \tau^{1/6}
    =
    \tau^{\frac{1}{12}}
    \le
    \tau^{\frac{\gamma}{18}/\log\frac{1}{\alpha}}
  \end{eqnarray*}
  where both inequalities uses that $\alpha \le \frac{1}{2}$ and the last also that $\gamma \le 1$.
\end{proof}

\subsection{Projective Lipschitz functions}
In our applications the test function $\Psi$ can be decomposed into
a projection $f_C$ onto some compact convex subset
$C \subseteq \Reals^k$
and a function a Lipschitz continuous function
$C \rightarrow \Reals$.
The projection $f_c:\Reals^k \rightarrow \Reals^k$ is defined by
$f(x)$ being the \emph{unique} point $y \in C$ which minimizes $\|x-y\|_2$.
The following standard lemma states that such projections are always Lipschitz.
\begin{lemma}
  \label{lem:convexProjIsLipschitz}
  Let $C \subseteq \Reals^k$ be a compact convex subset.
  Then $f_C$ is well-defined and Lipschitz continuous with Lipschitz
  constant 1.
\end{lemma}
\begin{proof}
  Let us first establish that $f_C$ is well-defined.
  Fix $x\in \Reals^k$.
  By compactness, there exists a $y$ which achieves
  $\inf_{y\in C} \|x-y\|_2$.
  For uniqueness, suppose $y'$ also achieves this, i.e.
  $\|x-y\|_2 = \|x-y'\|_2$.
  By convexity of $C$, $y^*:= \frac{y+y'}{2} \in C$.
  Still
  $\|x-y^*\|^2_2 = \|x-y\|^2_2-\frac{1}{2}\|y-y'\|^2_2$.
  But since $y$ minimizes $\|x-y\|_2$ we must have $y=y'$.

  Let us now turn to Lipschitz continuity.
  Fix $x,y\in \Reals^k$.
  We need to show that
  $\|f_C(y)-f_c(x)\|_2 \le \|y-x\|_2$.
  If $f_C(x)=f_C(y)$ we are done.
  Otherwise, let $L$ denote the line passing through $f_C(x)$ and
  $f_C(y)$
  and let $x_L$ and $y_L$ denote the orthogonal projection of $x$ and
  $y$ onto $L$.
  Clearly, $\|x_L-y_L\|_2 \le \|x-y\|_2$.
  By convexity, the intersection $C_L = C \cap L$ of $C$ and the line
  is a segment of $L$.
  It remains to show that $\|f(x)-f(y)\|_2 \le \|x_L-y_L\|_2$.
  But this is easy to see by considering three cases
  depending on whether $C_l$ and the
  segment $[f(x),f(y)]$ of the line $L$ are disjoint,
  one is contained in the other or they only partially overlap.
\end{proof}

\section{Proof of the Exchangeable Gaussians Theorem}
\label{sec:EGT}

\newcommand{\Ss}{\ensuremath{\stab_\Sigma}}
\newcommand{\Sst}{\ensuremath{\widetilde{\stab}_\Sigma}}
\newcommand{\Xt}{\widetilde{X}}
\newcommand{\Yt}{\widetilde{Y}}
\newcommand{\Zt}{\widetilde{Z}}
\newcommand{\mut}{\tilde \mu}

In this section we prove the EGT, Theorem \ref{thm:nGauss2SplitIsop}.
Our starting point will be the
\emph{extended Riesz inequality} on the sphere~\cite{Burchard:01, Morpurgo:02}.
Let $\SphereMo \subseteq \Reals^m$ be the $m\!\!-\!\!1$ -dimensional
sphere of radius $1$ in $\Reals^m$
and
for any Borel measurable set $A \subseteq \SphereMo$,
define its spherical rearrangement $A^*$ with respect to a point
$x^* \in \SphereMo$ as the spherical cap centered at $x^*$ with
the same measure as $A$, i.e.
$A^* = \{x:\|x-x^*\|_2 \le a\}$ for $a$ chosen so that
$A$ and $A^*$ has the same measure.

\begin{theorem} \cite[Theorem 3]{Burchard:01}
  \label{thm:Burchard}
  Fix $m \ge 1$
  and for $A_1, \ldots, A_k \in \Borel(\SphereMo)$ let
  \begin{equation}
    J_K(A_1, \ldots A_k)
    =
    \E
    \left[
      \prod_{i=1}^k 1_{\{X_i \in A_i\}}
      \prod_{1 \le i < j \le k} K_{i,j} (\|X_i-X_j\|)
    \right]
  \end{equation}
  where
  $K_{i,j}:\Reals_+ \rightarrow \Reals_+$ are non-increasing functions
  and $X_1, \ldots X_k$ are i.i.d. uniform on $\SphereMo$.
  Then,
  for any $A_1, \ldots, A_k \in \Borel(\SphereMo)$,
  \begin{equation}
    J_K(A_1, \ldots, A_k)
    \le
    J_K(A^*_1, \ldots, A^*_k)
  \end{equation}
  where $A^*_1, \ldots A^*_k$ are the spherical rearrangements of
  $A_1, \ldots, A_k$ with respect to some fixed point $x^* \in \SphereMo$.
\end{theorem}

We will prove a slightly more general version of
Theorem~\ref{thm:nGauss2SplitIsop} allowing for more general kinds of
noise in each dimension of the $k$ Gaussian vectors,
and different and possibly non-balanced sets for each vector.
In the rest of this section we will think of the vectors $X_1, \ldots X_k$ as
being column vectors in a matrix $X$, and we will write
$X_{.i}$ for the i'th row vector of $X$.

\begin{definition}
  \label{def:gaussNoiseStab}
  Let $\Sigma \in \Reals^{k \times k}$ be positive definite.
  Then the \emph{Gaussian $\Sigma$-noise stability} of
  $A_1, \ldots, A_k \in \Borel(\Reals^n)$ is
  \begin{equation}
    \Ss(A_1, \ldots, A_k) = \P(X_1 \in A_1, \dots, X_k \in A_k)
  \end{equation}
  where $X_{.1}, \ldots, X_{.n}$ are i.i.d. $\Norm(0,\Sigma)$.
  \newline
  We also let $\mu = \stab_{[1]}$ denote
  the standard Gaussian measure on $\Reals^n$.
\end{definition}

We will first prove a corresponding result
on the sphere from which
Theorem~\ref{thm:nGauss2SplitIsop}
can be derived based on Poincar\'{e}s observation that Gaussian
measure on $\Reals^n$ is
obtained by projection of the uniform measure on $\SphereMo$
onto $\Reals^n$, as $m \rightarrow \infty$.
Let us first define \emph{spherical $\Sigma$-noise stability}.

\begin{definition}
  Let $\Sigma \in \Reals^{k \times k}$ be positive definite.
  Then the \emph{spherical $\Sigma$-noise stability} of
  $A_1, \ldots, A_k \in \Borel(\SphereMo)$ is
  \begin{equation}
    \Sst(A_1, \ldots, A_k)
    =
    \P(\Xt_1 \in A_1, \dots, \Xt_k \in A_k)
  \end{equation}
  where $X_{.1}, \ldots, X_{.m}$ are i.i.d. $\Norm(0,\Sigma)$
  and $\Xt_i = \frac{X_i}{\|X_i\|_2}$.
  \newline
  We also let $\mut = \widetilde{\stab}_{[1]}$
  denote the uniform measure on the sphere $\SphereMo$.
\end{definition}

\begin{theorem}
  \label{thm:hOptSphere}
  Let $\Sigma \in \Reals^{k \times k}$ be positive definite
  with $\left(\Sigma^{-1}\right)_{i,j} \le 0$ for $i \neq j$.
  Then,
  for any $A_1, \ldots, A_k \in \Borel(\SphereMo)$,
  \begin{equation}
    \Sst(A_1, \ldots, A_k) \le \Sst(H_1, \ldots, H_k)
  \end{equation}
  where $H_i=\{x \in \SphereMo | x_1 \le a_i\}$ for
  $a_i$ chosen so that $\mut(H_i)=\mut(A_i)$.
\end{theorem}

\begin{proof}
  \begin{eqnarray*}
    \Sst(A_1, \ldots, A_k)
    &=&
    \P(\Xt_1 \in A_1, \ldots, \Xt_k \in A_k)
    \\
    &=&
    C
    \int_{\Reals^{km}}
    1_{\{\frac{x_1}{\|x_1\|_2} \in A_1, \ldots, \frac{x_k}{\|x_k\|_2} \in A_k \}}
    \prod_{l=1}^m
    e^{
      -\frac{1}{2}
      \sum_{1 \le i,j \le k}
      \left(\Sigma^{-1}\right)_{i,j} x_{il} x_{jl}
    }
    dx
  \end{eqnarray*}
  where $C=(2^k\pi^k |\Sigma|)^{-m/2}$.
  Changing the order of summation in the exponential we get
  \begin{eqnarray*}
    \Sst(A_1, \ldots, A_k)
    &=&
    C
    \int_{\Reals^{km}}
    1_{\{\frac{x_1}{\|x_1\|_2} \in A_1, \ldots, \frac{x_k}{\|x_k\|_2} \in A_k \}}
    \prod_{1 \le i<j \le k}
    e^{
      -
      \left(\Sigma^{-1}\right)_{i,j}
      x_i \cdot x_j
    }
    \prod_{i=1}^k
    e^{
      - \frac{1}{2}
      \left(\Sigma^{-1}\right)_{i,i}
      \|x_i\|^2_2
    }
    dx
    \\
    &=&
    C_1
    \E
    \left[
      1_{\{\frac{Z_1}{\|Z_1\|_2} \in A_1, \ldots, \frac{Z_k}{\|Z_k\|_2} \in A_k \}}
      \prod_{1 \le i<j \le k}
      e^{
        -
        \left(\Sigma^{-1}\right)_{i,j}
        Z_i \cdot Z_j
      }
    \right]
  \end{eqnarray*}
  where $Z_1, \ldots, Z_k$ are independent with
  $Z_i \in \Norm(0, \frac{1}{\left(\Sigma^{-1}\right)_{i,i}}I_m)$
  and $C_1 = (|\Sigma| \prod_{i=1}^k (\Sigma^{-1})_{i,i})^{-m/2}$.
  Conditioning on the lengths of the $Z_i$'s we have
  \begin{equation}
    \label{eq:expCondExp}
    \Sst(A_1, \ldots, A_k)
    =
    C_1 \E
    \left[
      \E
      \left[
        1_{\{\frac{Z_1}{\|Z_1\|_2} \in A_1, \ldots, \frac{Z_k}{\|Z_k\|_2} \in A_k \}}
        \prod_{1 \le i<j \le k}
        e^{
          -
          \left(\Sigma^{-1}\right)_{i,j}
          Z_i \cdot Z_j
        }
        \bigg|
        \|Z_1\|_2, \ldots, \|Z_k\|_2
      \right]
    \right]
  \end{equation}
  The inner conditional expectation can be expressed (almost
  surely with respect to the measure on the lengths) as
  \begin{equation}
    \label{eq:condExp}
    \E
    \left[
      1_{\{\Zt_1 \in A_1, \ldots, \Zt_k \in A_k \}}
      \prod_{1 \le i<j \le k}
      e^{
        -
        \left(\Sigma^{-1}\right)_{i,j}
        \|Z_i\|_2 \|Z_j\|_2
        \Zt_i \cdot \Zt_j
      }
    \right]
  \end{equation}
  where $\Zt_1, \ldots, \Zt_k$ are i.i.d uniform on $\SphereMo$.
  But since $\Zt_i \cdot \Zt_j$ is decreasing in $\|\Zt_i - \Zt_j\|$
  and $\left(\Sigma^{-1}\right)_{i,j} \le 0$ for $i\neq j$,
  Theorem~\ref{thm:Burchard} implies that replacing each $A_i$ with $A_i^*$ in
  \eqref{eq:condExp} will not decrease the value of
  \eqref{eq:condExp} and hence not the value of \eqref{eq:expCondExp}.
  Thus, $\Ss(A_1, \ldots, A_k) \le \Ss(A_1^*, \ldots, A_k^*) =
  \Ss(H_1, \ldots, H_k)$.
\end{proof}

We are now ready to prove the more general version of
Theorem~\ref{thm:nGauss2SplitIsop}
allowing for more general noise as well as for sets of arbitrary fixed measure.

\begin{theorem}
  \label{thm:hOpt}
  Let $\Sigma \in \Reals^{k \times k}$ be positive definite
  with $\left(\Sigma^{-1}\right)_{i,j} \le 0$ for $i \neq j$.
  Then,
  for any $A_1, \ldots, A_k \in \Borel(\Reals^n)$,
  \begin{equation}
    \Ss(A_1, \ldots, A_k) \le \Ss(H_1, \ldots, H_k)
  \end{equation}
  where $H_i=\{x \in \Reals^n | x_1 \le a_i\}$ for
  $a_i$ chosen so that $\mu(H_i)=\mu(A_i)$.
\end{theorem}

\begin{proof}
  For fixed $m \ge n$,
  let $X_{.1}, \ldots, X_{.m}$ be i.i.d. $\Norm(0,\Sigma)$
  and $\Xt_i = \frac{X_i}{\|X_i\|_2}$.

  Further, let $Y_i = (X_{i,1}, \dots, X_{i,n})$ denote the restriction of
  $X_i$ to the first $n$ coordinates, and similarly
  $\Yt_i = (\Xt_{i,1}, \dots, \Xt_{i,n})$.
  Then, the central limit theorem implies
  \begin{equation}
    (\Yt_1, \dots, \Yt_k)
    \stackrel{\mathcal{D}}{\rightarrow}
    (Y_1, \dots, Y_k)
    \text{ as }
    m \rightarrow \infty
  \end{equation}

  Suppose first that $A_1, \ldots, A_k$ are closed sets.
  Then, by~\cite[Theorem 2.4]{Durrett:05},
  \begin{equation}
    \limsup_{m \rightarrow \infty}
    \P
    \left(
      (\Yt_1,\ldots,\Yt_k) \in
      \prod_{i=1}^k A_i
    \right)
    \le
    \P
    \left(
      (Y_1,\ldots,Y_k) \in
      \prod_{i=1}^k A_i
    \right)
  \end{equation}
  i.e.
  \begin{equation}
    \label{eq:hOptSphere1}
    \limsup_{m \rightarrow \infty}
    \Sst(A_1 \times \Reals^{m-n}, \ldots, A_k \times \Reals^{m-n})
    \le
    \Ss(A_1, \ldots, A_k)
  \end{equation}
  $H$ on the other hand has a boundary of zero measure, so a similar
  application of~\cite[Theorem 2.4]{Durrett:05} gives
  \begin{equation}
    \label{eq:hOptSphere2}
    \limsup_{m \rightarrow \infty}
    \Sst(H_1 \times \Reals^{m-n}, \ldots, H_k \times \Reals^{m-n})
    =
    \Ss(H_1, \ldots, H_k)
  \end{equation}
  But by Theorem~\ref{thm:hOptSphere},
  \begin{equation}
    \label{eq:hOptSphere3}
    \Sst(A_1 \times \Reals^{m-n}, \ldots, A_k \times \Reals^{m-n})
    \le
    \Sst(H_1 \times \Reals^{m-n}, \ldots, H_k \times \Reals^{m-n})
  \end{equation}
  hence,
  Combining \eqref{eq:hOptSphere1}, \eqref{eq:hOptSphere2},
  \eqref{eq:hOptSphere3} gives the result for closed $A_1, \ldots, A_k$.

  If not all $A_i$'s are closed, regularity of the uniform measure $\mut$
  implies that for all $\epsilon>0$ and $i\in [k]$ there exist
  closed $A_i' \subseteq A_i$ such that
  $\mu(A_i') \ge \mu(A_i)-\epsilon$, and hence
  $\Ss(A_1', \ldots, A_k') \ge \Ss(A_1, \ldots, A_k) -k \epsilon$,
  and the result follows from
  the result for closed sets by letting $\epsilon \rightarrow 0$.
\end{proof}

The EGT now follows as a special case of Theorem~\ref{thm:hOpt},

\begin{proof}[Proof of Theorem \ref{thm:nGauss2SplitIsop}]
  Take
  $\Sigma_{i,j} = \rho + (1-\rho) \delta_{ij}$
  in Theorem~\ref{thm:hOpt}
  and note that
  the inverse of $\Sigma$ is given by
  $\left(\Sigma^{-1}\right)_{i,j} = -a + b \delta_{ij}$, where
  $b = \frac{1}{1-\rho}$ and
  $a = \frac{\rho}{(1-\rho)(1+\rho(k-1))} \ge 0$ for $\rho \ge 0$.
\end{proof}

{\bf Remark:}
As pointed out to us by an anonymous referee, the special case of
the EGT where all sets $A_i$ are identical (and all pairs $X_i$,
$X_j$ have the same covariance $\rho \ge 0$) also follows from
Borell's result \cite{Borell:85}. See also \cite[Theorem 4.1]{MoOdOl:09}.

\section{Exchangeable low influence bounds}
\label{sec:lowInflBounds}

Combining the EGT and the invariance principle allows us to derive
stability bounds on discrete low-influence functions.
In this section we derive a general bound on the stability of discrete
low-influence functions which is used for the applications in the next section.

Letting $\Sigma(X)$ denote the $\sigma$-algebra generated by $X$ we show,

\begin{theorem}
  \label{thm:discreteEGT}
  Fix $\rho \in [0,1]$ and
  let $X_1, \ldots, X_k \in \Omega^n$ be random column vectors such that the
  row vectors $X_{.1}, \ldots, X_{.n}$ are i.i.d.
  with
  $
  \rho(\Sigma\left(X_{1,1}\right), \dots, \Sigma\left(X_{k,1}\right); \P)
  = \rho' < 1
  $
  and $X_1, \ldots, X_k$ are pairwise $\rho$-correlated
  in that for any $j_1 \neq j_2$,
  \begin{equation}
    \label{eq:pairwiseRhoCorr}
    \P(X_{j_1}=\lambda|X_{j_2}=\omega)
    =
    \prod_{i=1}^n
    \left(
    \rho 1_{\{\lambda_i=\omega_i\}} + (1-\rho)\P(X_{j_1,i}=\lambda_i)
    \right)
  \end{equation}
  Then,
  for any $\epsilon>0$
  there exists a $\tau(\epsilon, k, \rho')>0$
  such that for any $f_1, \ldots, f_k:\Omega^n \rightarrow [0,1]$
  satisfying $\max_{i,j} \Inf_i f_j \le \tau$,
  \begin{equation}
    \label{eq:ppEnoughII}
    \E \prod_{j=1}^k f_j(X_j)
    \le
    \P(\forall j: Z_j \le a_j)
    +
    \epsilon
  \end{equation}
  where
  $Z_1,\dots,Z_k \sim \Norm(0, 1)$ are jointly normal with
  $\Cov(Z_i,Z_j)=\rho$ for $i\neq j$,
  and each $a_j$ is chosen so that $\P(Z_1 \le a_j) = \E f_j(X_j)$.
\end{theorem}

\begin{proof}
  Let $m=|\Omega|-1$, $M=nm$
  and $G_1, \ldots, G_k \in \Norm(0, I_M)$ be jointly normal
  with $\Cov(G_{j_1}, G_{j_2}) = \rho I_M$ for $j_1 \neq j_2$.

  Note that all variables $X_{j,i}$ have the same marginal measure.
  Thus, we can fix an orthonormal basis
  $\calV(x)=\{V_0(x) = 1, V_1(x), \dots, V_{m}(x)\}$
  for functions
  $\Omega \rightarrow \Reals$
  under this marginal measure
  and
  form orthonormal ensembles
  \begin{eqnarray*}
    \calX_i^j =& (1, V_1(X_{j,i}), \ldots, V_m(X_{j,i}))
    &\text{ , for } i \in [n] \text{ and } j \in [k]
    \text{ , and }
    \\
    \calG_i^j =& (1, G_{j,m(i-1)+1}, \ldots, G_{j,m(i-1)+m})
    &\text{ , for } i \in [n] \text{ and } j \in [k]
  \end{eqnarray*}
  and independent sequences of orthonormal ensembles
  \begin{eqnarray*}
    \calX^{j}=(\calX_1^j,\dots,\calX_n^j)
    \text{ and }
    \calG^{j}=(\calG_1^j,\dots,\calG_n^j)
    \text{ for } j \in [k]
  \end{eqnarray*}
  Then $\calX^{j}$ is a basis for all real-valued functions on $X_j$
  and we can compute the
  (unique) multilinear polynomial $Q_j$ such that
  $f_j(X_j) = Q_j(\calX^j)$.
  Hence we may write,
  \begin{equation}
    \label{eq:uniqBestExp}
    \E \prod_{j=1}^k f_j(X_j)
    =
    \E\prod_{j=1}^k Q_j(\calX^j)
  \end{equation}
  For each $j$
  let $\Qtilde_j = T_{1-\gamma}Q_j$ be a slightly smoothed version of $Q_j$.
  Since $\rho' < 1$, by Lemma \ref{lem:smoothing} we can find a
  $\gamma=\gamma(\epsilon, k, \rho')>0$ such that
  \begin{equation}
    \label{eq:ppSmoothII}
    \left|
    \E\prod_{j=1}^k Q_j(\calX^j)
    -
    \E\prod_{j=1}^k \Qtilde_j(\calX^j)
    \right|
    \le
    \frac{\epsilon}{2k}
  \end{equation}
  Let $f_{[0,1]}(x) = \max(0,\min(1,x))$.
  Since $Q_j$ has range $[0,1]$,
  the same holds for $\Qtilde_j$. Hence, for all $j$,
  \begin{equation}
    \label{eq:ppEqII}
    \Qtilde_j(\calX^j)
    =
    f_{[0,1]}\Qtilde_j(\calX^j)
  \end{equation}
  Now form new ensembles,
  \begin{eqnarray*}
    \calX_i=&
    (1,
    V_1(X_{1,i}), \ldots, V_m(X_{1,i}),
    \ldots,
    V_1(X_{k,i}), \ldots, V_m(X_{k,i})
    )
    \text{ for } i \in [n]
    \text{ , and }
    \\
    \calG_i=&
    (1,
    G_{1,m(i-1)+1}, \ldots, G_{1,m(i-1)+m},
    \ldots,
    G_{k,m(i-1)+1}, \ldots, G_{k,m(i-1)+m}
    )
    \text{ for } i \in [n]
  \end{eqnarray*}
  and note that
  $\calX=(\calX_1, \ldots, \calX_n)$
  and
  $\calG=(\calG_1, \ldots, \calG_n)$
  are two independent sequences of ensembles with
  a matching covariance structure, since
  by \eqref{eq:pairwiseRhoCorr}, for $j_1 \neq j_2 \in [k]$,
  \begin{equation}
    \E[V_{l_1}(X_{j_1,i}) V_{l_2}(X_{j_2,i})]
    =
    \rho 1_{\{l_1=l_2\}}
    =
    \E[G_{j_1,m(i-1)+l_1} G_{j_2,m(i-1)+l_2}]
  \end{equation}
  Further, $\Inf_i Q_j (\calX) = \Inf_i f (X_j)$.
  Hence, we may apply the invariance principle
  (Theorem \ref{thm:invariancePrincipleFinal})
  on the $k$-dimensional multilinear polynomial
  \begin{equation}
    Q(\calX) = Q(\calX_1, \ldots, \calX_n)
    =
    [Q_1(\calX^1) \ldots Q_k(\calX^k)]^T
  \end{equation}
  using
  $\Psi(x_1, \dots, x_k) = \prod_{j=1}^k f_{[0,1]}(x_j)$ which
  is Lipschitz by convexity
  of $[0,1]^{k-1}$, Lemma~\ref{lem:convexProjIsLipschitz}
  and the product $\prod_{j=1}^k x_j$ being Lipschitz on $[0,1]^{k-1}$.
  Thus, by Theorem \ref{thm:invariancePrincipleFinal},
  there exist some $\tau>0$ such that,
  \begin{equation}
    \label{eq:ppInvII}
    \left|
    \E\prod_{j=1}^k f_{[0,1]}\Qtilde_j(\calX^j)
    -
    \E\prod_{j=1}^k f_{[0,1]}\Qtilde_j(\calG^j)
    \right|
    \le
    \frac{\epsilon}{4k^2}
  \end{equation}
  Now $(f_{[0,1]}\Qtilde,1 \! - \! f_{[0,1]}\Qtilde)$ applied to $\calG^j$ can be thought of as a function $\Reals^{n} \rightarrow \Delta_2$ creating a fuzzy partition of the $n$-dimensional Gaussian space.
  Let $\mu_j = \E f_j (X_j) = \E f_{[0,1]}\Qtilde_j (\calX^j)$.
  Then a second application of Theorem \ref{thm:invariancePrincipleFinal} with $\Psi(x) = f_{[0,1]}(x)$ gives
  \begin{equation}
    \left|
      \E f_{[0,1]}\Qtilde_j (\calG^j) - \mu_j
    \right|
    \le
    \frac{\epsilon}{4k^2}
  \end{equation}
  By Lemma \ref{lem:nonFuzziness} and \ref{lem:almostBalanced},
  there exist functions $g_1, \ldots, g_k:\Reals^{M} \rightarrow E_2$
  with $\E g_{j,1}(G_j) = \mu_j$ and
  \begin{equation}
    \label{eq:ppFuzzyII}
    \E\prod_{j=1}^k f_{[0,1]}\Qtilde(\calG^j)
    \le
    \E\prod_{j=1}^k g_{j,1}(G_j)
    + \frac{\epsilon}{4k}
  \end{equation}
  But any such $g_j$ partitions $\Reals^{M}$ into $2$ parts of
  of measure $\mu_j$ and $1-\mu_j$ respectively, so
  Theorem~\ref{thm:nGauss2SplitIsop}
  implies
  \begin{equation}
    \label{eq:ppConjII}
    \E\prod_{j=1}^k g_{j,1}(\calG^j)
    \le
    \P(\forall j: G_j \in H_j)
    =
    \P(\forall j: Z_j \le a_j)
  \end{equation}
  where $H_j = \{x \in \Reals^M| x_1 \le a_j\}$.
  Combining equations \eqref{eq:uniqBestExp}, \eqref{eq:ppSmoothII},
  \eqref{eq:ppEqII},
  \eqref{eq:ppInvII}, \eqref{eq:ppFuzzyII} and
  \eqref{eq:ppConjII}
  gives \eqref{eq:ppEnoughII} as needed.
\end{proof}

\section{Applications of the EGT}
\label{sec:appEGT}
In this section we show the two applications of the EGT
in Condorcet voting and Cosmic coin flipping
using the influence bounds proved in Section~\ref{sec:lowInflBounds}.

\subsection{Condorcet voting}
\label{sec:appCondorcet}
Here we use Theorem \ref{thm:nGauss2SplitIsop} to show that
majority maximizes the probability of having a
unique best candidate in Condorcet voting (Theorem
\ref{thm:condorcet}).

Remember that we have $n$ voters selecting a linear order $\sigma_i
\in S(k)$ uniformly at random and let
\begin{equation}
 X_i^{a>b}=
 \left\{
   \begin{array}{ll}
     1 & \mbox{ if } \sigma_i(a) > \sigma_i(b) \\
     -1 & \mbox{ else }
   \end{array}
 \right.
 \text{, for } i \in [n] \text{ and } a,b \in [k].
\end{equation}
By considering the 6 possible linear orders of three candidates its
easy to see that for any distinct $a,b,c \in [k]$ we have
\begin{equation*}
  \E[X_i^{a>b}] = 0
  \text{ , }
  \Var X_i^{a>b} = 1
  \text{ and }
  \Cov[X_i^{a>b}, X_i^{a>c}] = \frac{1}{3}
\end{equation*}

First we will show that the limit of the probability of having a
unique best candidate using the majority function
corresponds to the right hand side of \eqref{eq:nGauss2SplitIsop}.
\begin{lemma}
  \label{lem:UniqueBestEqConj}
  Let $Z_2,\dots,Z_k \sim \Norm(0, I_n)$ be jointly normal with
  $\Cov(Z_i,Z_j)=\frac{1}{3} I_n$ for $i\neq j$.
  Then
  \begin{equation}
    \lim_{n \rightarrow \infty} \P[\UniqueBest_k(\MAJ_n)]
    =
    \P(Z_2 \in H, \ldots, Z_k \in H)
\end{equation}
where $H =\{x \in \Reals^n | x_1 \le 0\}$.
\end{lemma}

\begin{proof}
  Let $Y_j=\frac{1}{\sqrt n} \sum_{i=1}^n X_i^{1>j}$.
  By definition \ref{def:uniqueBest},
  \begin{equation}
    \P[\UniqueBest_k(\MAJ_n)]
    =
    \P(Y_2 \ge 0, \dots, Y_k \ge 0)
  \end{equation}
  But, $\E[Y_j]=0$, $\E[Y_j^2] = 1$ and $\Cov[Y_i,Y_j] = \frac{1}{3}$ for
  $i \neq j$. Thus, by the central limit theorem,
  $(Y_2, \dots, Y_k) \stackrel{\mathcal{D}}{\rightarrow} (Z_2, \dots, Z_k)$ and the result follows.
\end{proof}

\begin{proof}[Proof of Theorem \ref{thm:condorcet}]
  Clearly, any candidate has the same probability of being the unique
  best candidate. So it's enough to show that the probability that the
  first candidate is the unique best is maximized by majority,
  i.e. for some $\tau$ small enough,
  \begin{equation}
    \P[\UniqueBest_k(f,1)]
    \le
    \lim_{n \rightarrow \infty}
    \P[\UniqueBest_k(\MAJ_n,1)] + \frac{\epsilon}{k}
  \end{equation}
  But,
  \begin{equation}
    \P[\UniqueBest_k(f,1)]
    =
    \P(f(X^{1>2}) = \ldots = f(X^{1>k}) = 1)
  \end{equation}
  Let
  \begin{equation}
    \rho(k) = \rho(\Sigma\left(X_i^{1>2}\right), \dots, \Sigma\left(X_i^{1>k}\right); \P)
  \end{equation}
  To see that $\rho(k)<1$ it is by symmetry enough to show that
  $
  \rho(\Sigma\left(X_i^{1>2}\right), \Sigma\left(X_i^{1>3},
    \ldots, X_i^{1>k}\right); \P)<1
  $.
  But this follows by Lemma~\ref{lem:connRho}, since the bipartite
  graph of Lemma~\ref{lem:connRho} is complete
  and any edge has a probability of at least
  $\frac{1}{k!}$ since it occurs in at least one ordering.

  Hence, by applying Theorem~\ref{thm:discreteEGT}
  on the vectors $X^{1>2}, \ldots, X^{1>k}$
  using $f$ as $f_1, \ldots, f_k$
  we can find a $\tau=\tau(\eps,k)$ such that
  \begin{equation}
    \P(f(X^{1>2}) = \ldots = f(X^{1>k}) = 1)
    \le
    \P(\forall i: Z_i \in H) + \epsilon
  \end{equation}
  Lemma~\ref{lem:UniqueBestEqConj} now gives the result.
\end{proof}

\subsection{Cosmic coin flipping}
\label{sec:appCosmicCoin}
Here we use Theorem \ref{thm:nGauss2SplitIsop} to show that
majority maximizes the probability of all players agreeing
in cosmic coin flipping (Theorem \ref{thm:cosmicCoin}).
Remember that we want to maximize
\begin{equation}
  \mathcal{P}^{(k,n)}_\rho(f) := \P(f(Y_1) = \ldots = f(Y_k))
\end{equation}
where each $Y_i$ is uniform on $\{-1,1\}^n$ and
$\Cov[Y_i,Y_j] = \rho^2 I_n$ for $i \neq j$.

First we will show that the limit of the probability of all players agreeing
when using the majority function
corresponds to twice the right hand side of \eqref{eq:nGauss2SplitIsop}.
\begin{lemma}
  \label{lem:cosmicCoinMajEq}
  Let $Z_1,\dots,Z_k \sim \Norm(0, I_n)$ be jointly normal with
  $\Cov(Z_i,Z_j)=\rho^2 I_n$ for $i\neq j$.
  Then
  \begin{equation}
    \lim_{n \rightarrow \infty}
    \mathcal{P}^{(k,n)}_\rho (\MAJ_n)
    =
    2 \P(\forall j: Z_j \in H)
\end{equation}
where $H =\{x \in \Reals^n | x_1 \le 0\}$.
\end{lemma}

\begin{proof}
  Let $X_j=\frac{1}{\sqrt n} \sum_{i=1}^n Y_{j,i}$.
  Then,
  \begin{equation}
    \mathcal{P}^{(k,n)}_\rho (\MAJ_n)
    =
    \P(X_1 \ge 0, \dots, X_k \ge 0)
    +
    \P(X_1 < 0, \dots, X_k < 0)
  \end{equation}
  But, $\E[X_j]=0$, $\E[X_j^2] = 1$ and $\Cov[X_i,X_j] = \rho^2$ for
  $i \neq j$. Thus, by the central limit theorem,
  $(X_1, \dots, X_k) \stackrel{\mathcal{D}}{\rightarrow} (Z_{1,1}, \dots, Z_{k,1})$
  and the result follows.
\end{proof}

\begin{proof}[Proof of Theorem \ref{thm:cosmicCoin}]
  The theorem is trivial for $\rho=1$.
  So assume $\rho \in [0,1)$.
  To see that
  \begin{equation}
    \rho(\Sigma\left(Y_1\right), \dots, \Sigma\left(Y_k\right); \P)
    < 1
  \end{equation}
  it is by symmetry enough to show that
  $
  \rho(\Sigma\left(Y_1\right), \prod_{j=2}^k \Sigma\left(Y_j\right); \P)
  < 1
  $.
  But this follows from Lemma~\ref{lem:connRho},
  since every value of $Y$ occurs with non-zero probability
  and hence the bipartite graph of Lemma~\ref{lem:connRho} is connected
  and by finiteness, the minimal probability of an edge depend only on $k$ and $\rho$.
  Hence, by applying Theorem~\ref{thm:discreteEGT} on
  the variables $Y_1, \ldots, Y_k$ twice,
  first using $f$ as $f_1, \ldots, f_k$ and then using $1-f$,
  we can find a $\tau=\tau(\eps,k,\rho)$ such that
  \begin{equation}
    \mathcal{P}^{(k,n)}_\rho(f)
    =
    \E \prod_{j=1}^k f(Y_j)
    +
    \E \prod_{j=1}^k (1-f(Y_j))
    \le
    2 \P(\forall j: Z_j \in H) + \epsilon
  \end{equation}
  Lemma~\ref{lem:cosmicCoinMajEq} now gives the result.
\end{proof}

\section{Plurality is stablest}
\label{sec:appPlur}

Here we show that Conjecture \ref{conj:2GaussQSplitIsop} implies the
Plurality is stablest conjecture (Theorem \ref{thm:plurStabAppl}).

We start by showing an unconditional result that asserts that the
most stable low low-degree influence functions are essentially determined by
most stable partition of Gaussian space into $q$ parts
of equal measure.

\begin{theorem}
  \label{thm:lowInflStabToGaussStab}
  For any $q \ge 2$, $\rho \in [-\frac{1}{q-1},1]$
  and $\epsilon>0$ there exist $d$ and $\tau>0$ such that
  if $f:[q]^n \rightarrow \Delta_q$
  has $\Inf_i^{\le d}(f_j) \le \tau$, $\forall i,j$,
  then there exists a $g:\Reals^{n(q-1)} \rightarrow E_q$
  such that $\E g = \E f$ and
  \begin{equation}
    \left|
      \Sp(g)
      -
      \Sp(f)
    \right|
    \le \epsilon
  \end{equation}
\end{theorem}

\begin{definition}
  For $q \ge 2$,
  let $f_{\Delta_q}:\Reals^q \rightarrow \Delta_q$ denote the
  function
  which maps $x$ to the point in $\Delta_q$ which is closest to $x$.
\end{definition}

\begin{proof}[Proof of Theorem \ref{thm:lowInflStabToGaussStab}]
  The result is trivial for $\rho=1$ so assume $\rho \in [-\frac{1}{q-1},1)$.
  Let $(\Omega \times \Lambda, \mu)$, with the $\rho$-correlated measure
  $\mu(\omega, \lambda)=\rho 1_{\{\lambda=\omega\}} \frac 1 q + (1-\rho) \frac{1}{q^2}$
  be our base space and let
  $(\omega,\lambda)\in [q]^n \times [q]^n$ be drawn from $\mu^n$.

  Fix an orthonormal basis
  $\calV(x)=\{V_0(x) = 1, V_1(x), \dots, V_{q-1}(x)\}$
  for functions
  $[q] \rightarrow \Reals$
  and construct two sequences of orthonormal ensembles
  $\calX = \{\calX_1, \dots, \calX_n\}$ and
  $\calY = \{\calY_1, \dots, \calY_n\}$
  for functions $\Omega \rightarrow \Reals$ and $\Lambda \rightarrow
  \Reals$
  by letting
  $X_{i,j}(\omega_i) = V_{j}(\omega_i)$ and
  $Y_{i,j}(\lambda_i) = V_{j}(\lambda_i)$.
  Note that this means that
  \begin{equation}
    \Cov(X_{i_1,j_1}, Y_{i_2,j_2}) =
    \left\{
      \begin{array}[]{ll}
        \rho & \mbox{ if } i_1=i_2 \mbox{ and } j_1=j_2 > 0 \\
        0 & \mbox{ else }
      \end{array}
    \right.
  \end{equation}
  Expressing $f(x_1, \dots, x_n)$ as a q-dimensional multi-linear polynomial
  $Q(\calV(x_1), \dots, \calV(x_n))$ we get
  \begin{equation}
    \label{eq:ppExp}
    \Sp(f) = \sum_{i=1}^q \E[f_j(\omega)f_j(\lambda)]
    = \sum_{i=1}^q \E[Q_j(\calX)Q_j(\calY)]
  \end{equation}

  Let $\Qtilde=T_{1-\gamma}Q$ be a slightly smoothed version of $Q$.
  To show that $\rho(\Omega,\Lambda; \mu) < 1$, observe that
  $\rho(\Omega,\Lambda; \mu)$ is given by the supremum of
  $\E[f_1(\omega)f_2(\lambda)]$
  over all
  $f_1:\Omega\rightarrow \Reals$
  and
  $f_2:\Lambda\rightarrow \Reals$
  such that $Ef_i=0$ and $Ef_i^2=1$.
  But
  $
  \E[f_1(\omega)f_2(\lambda)]
  =
  \E[f_1(\omega) (\rho f_2(\omega) + (1-\rho) \E f_2)]
  =
  \rho \E[f_1(\omega)f_2(\omega)]
  \le
  |\rho|
  $ by Cauchy-Schwarz.
  Hence, $\rho(\Omega,\Lambda; \mu) < |\rho| < 1$
  and by Lemma \ref{lem:smoothing} we can find a
  $\gamma(\epsilon,\rho,q)>0$ s.t.
  \begin{equation}
    \label{eq:ppSmooth}
    \left|
      \E[Q_j(\calX)Q_j(\calY)]
      -
      \E[\Qtilde_j(\calX)\Qtilde_j(\calY)]
    \right|
    \le
    \frac{\epsilon}{2q}
  \end{equation}
  Since $Q(\calX)$ has range $\Delta_q$, the same holds for
  $\Qtilde(\calX)$. Hence,
  \begin{equation}
    \label{eq:ppEq}
    f_{\Delta_q}\Qtilde(\calX) = \Qtilde(\calX)
  \end{equation}
  (and similarly for $\calY$).
  We are now ready to apply the invariance principle (Theorem \ref{thm:invariancePrincipleFinal})
  using
  $\Psi(x,y) = \ip{f_{\Delta_q}(x)}{f_{\Delta_q}(y)}$.
  To see that $\Psi(x,y)$ is Lipschitz continuous
  note that $f_{\Delta_q}$ is Lipschitz
  by convexity of $\Delta_q$ and Lemma~\ref{lem:convexProjIsLipschitz},
  and the inner product $\ip{x}{y} = \sum_{i=1}^k x_i y_i$
  is Lipschitz on $\Delta_q^2$.
  Hence Theorem \ref{thm:invariancePrincipleFinal} implies that for some $\tau>0$ small enough,
  \begin{equation}
    \label{eq:ppInv}
    \left|
      \E[\ip{f_{\Delta_q}\Qtilde(\calX)}{f_{\Delta_q}\Qtilde(\calY)}]
      -
      \E[\ip{f_{\Delta_q}\Qtilde(\calG)}{f_{\Delta_q}\Qtilde(\calH)}]
    \right|
    \le
    \frac{\epsilon}{4 q}
  \end{equation}
  where $\calG$ and $\calH$ are two Gaussian sequences of orthonormal ensembles
  with
  \begin{equation}
    \Cov(G_{i_1,j_1}, H_{i_2,j_2}) =
    \left\{
      \begin{array}[]{ll}
        \rho & \mbox{ if } i_1=i_2, j_1=j_2 > 0\\
        0 & \mbox{ else }
      \end{array}
    \right.
  \end{equation}

  $f_{\Delta_q}\Qtilde$ applied to $\calG$ or $\calH$ can be thought
  of as a function $\Reals^{n(q-1)} \rightarrow \Delta_q$ creating a
  fuzzy partition of the $n(q-1)$-dimensional Gaussian space.
  The balance of this partition might not equal the balance of $f$
  though. In particular,
  \begin{equation}
    \E f = \E Q(\calX) =
    \E f_{\Delta_q}\Qtilde(\calX)
    \neq
    \E f_{\Delta_q}\Qtilde(\calG)
  \end{equation}
  But applying Theorem \ref{thm:invariancePrincipleFinal} again,
  using $\Psi(x) = f_{\Delta_q,j} (x)$ which by
  Lemma~\ref{lem:convexProjIsLipschitz} is Lipschitz continuous with
  $A=1$, we can bound the total variation distance by
  \begin{equation}
    \sum_{j=1}^q
    |\E f_{\Delta_q,j}\Qtilde(\calG) - \E f_{\Delta_q,j}\Qtilde(\calX)|
    \le
    q \frac{\epsilon}{4 q}
    =
    \frac{\epsilon}{4}
  \end{equation}
  Hence, by Lemma \ref{lem:nonFuzziness} and \ref{lem:almostBalanced}
  there exists a function
  $g:\Reals^{n(q-1)} \rightarrow E_q$ such that $\E g = \E f$ and
  \begin{equation}
    \label{eq:ppFuzzy}
    \left|
    \E[\ip{f_{\Delta_q}\Qtilde(\calG)}{f_{\Delta_q}\Qtilde(\calH)}]
    -
    \E[\ip{g(\calG)}{g(\calH)}]
    \right|
    \le
    \frac{\epsilon}{4}
  \end{equation}
  But $\E[\ip{g(\calG)}{g(\calH)}] = \Sp(g)$, hence
  combining equations \eqref{eq:ppExp}, \eqref{eq:ppSmooth}, \eqref{eq:ppEq},
  \eqref{eq:ppInv} and \eqref{eq:ppFuzzy}
  gives the desired result.
\end{proof}

In order to prove Theorem \ref{thm:plurStabAppl} we first show that the limit of the noise stability of $\PLUR{n,q}$ corresponds to the right hand side of \eqref{eq:2GaussQSplitIsop}.

\begin{lemma}
  \label{lem:PLUReqConj}
  Fix $\rho \in [-\frac{1}{q-1},1]$ and $q \ge 3$.
  Let $X,Y \sim \Norm(0, I_{q-1})$ and $\Cov(X,Y)=\rho I_{q-1}$.
  Then
  \begin{equation}
    \lim_{n \rightarrow \infty} \Sp(\PLUR_{n,q})
    =
    \P((X,Y) \in S_1^2 \cup \dots \cup S_q^2)
  \end{equation}
  where $S_1, \ldots, S_q$ is a standard simplex partition of $\Reals^{q-1}$.
\end{lemma}

\begin{proof}
  Let $S_1, \ldots, S_q$ be the standard simplex partition determined
  by the unit vectors $a_1, \ldots, a_q \in \Reals^{q-1}$ according to
  Definition \ref{def:standardSimplex}.
  By Definition \ref{def:SpQ},
  \begin{equation}
    \Sp(\PLUR_{n,q})
    =
    \E\ip{\PLUR_{n,q}(\omega)}{\PLUR_{n,q}(\lambda)}
  \end{equation}
  where $\omega, \lambda$ are uniform on $[q]^n$ and satisfy
  \eqref{eq:defSpQ}.
  Let
  \begin{equation}
    V(\omega)=\sqrt{\frac{q-1}{q}} \frac{1}{\sqrt{n}} \sum_{i=1}^n
    a_{\omega_i}
  \end{equation}
  Then, conditioning on having no ties which will happen with
  probability $1$ as $n \rightarrow \infty$, we have
  \begin{equation}
    \PLUR_{n,q}(\omega)=j
    \iff
    \left[
    \forall i \neq j:
    V(\omega) \cdot a_j
    >
    V(\omega) \cdot a_i
    \right]
    \iff
    V(\omega) \in S_j
  \end{equation}
  and
  \begin{equation}
    \label{eq:plurSimp1}
    \Sp(\PLUR_{n,q})
    =
    \P[(V(\omega),V(\lambda)) \in S_1^2 \cup \dots \cup S_q^2]
  \end{equation}
  Now,
  \begin{eqnarray*}
    &&\E[a_{\omega_i}]=0, \text{ since }
    \E[a_{\omega_i}] \cdot a_j = 1-(q-1)\frac{1}{q-1} = 0
    \\
    &&\E[a_{\omega_i} a_{\omega_i}^T]=\frac{q}{q-1} I_{q-1},
    \text{ since }
    \E[a_{\omega_i} a_{\omega_i}^T ] \cdot a_j
    = a_j + \sum_{k \neq j} a_k \frac{-1}{q-1}
    = \frac{q}{q-1} a_j
    \\
    &&\E[a_{\omega_i} a_{\lambda_i}^T]
    =
    \E[a_{\omega_i} \E[a_{\lambda_i}^T| \omega_i]]
    =
    \E[\rho a_{\omega_i} a_{\omega_i}^T] =
    \rho \frac{q}{q-1} I_{q-1},
  \end{eqnarray*}

  Hence, by the central limit theorem, $(V(\omega),V(\lambda))$
  converges to a normal distribution with the same parameters as $(X,
  Y)$. Thus,
  \begin{equation}
    \lim_{n\rightarrow \infty}
    \P((V(\omega), V(\lambda)) \in S_1^2 \cup \dots \cup S_q^2)
    =
    \P((X,Y) \in S_1^2 \cup \dots \cup S_q^2)
  \end{equation}
  which together with \eqref{eq:plurSimp1} gives the result.
\end{proof}

\begin{proof}[Proof of Theorem \ref{thm:plurStabAppl}]
  Fix $q \ge 2$, $\rho \in [-\frac{1}{q-1},1]$
  and $\epsilon>0$ and let $\tau$ and $d$ be the
  constants given by Theorem \ref{thm:lowInflStabToGaussStab}.
  For any $f:[q]^n \rightarrow \Delta_q$ with
  $\Inf^{\le d}_i(f_j) \le \tau$, $\forall i,j$,
  we can thus find a $g:\Reals^{n(q-1)}$ such that
  $\E g = \E f$ and $|\Sp(g)-\Sp(f)| \le \eps$.
  But Conjecture \ref{conj:2GaussQSplitIsop} and Lemma
  \ref{lem:PLUReqConj} implies that
  \begin{align}
    &\Sp(g) \le \lim_{n \rightarrow \infty} \Sp(\PLUR_{n,q})
    \hspace{0.4cm} \mbox { if } \rho \ge 0
    \mbox{ and } g \mbox{ is balanced }
    \shortintertext{and}
    & \Sp(g) \ge \lim_{n \rightarrow \infty} \Sp(\PLUR_{n,q})
    \hspace{0.4cm} \mbox { if } \rho \le 0
  \end{align}
  which gives the slightly stronger result that only requires
  $\Inf^{\le d}_i f_j $ small.
  That we may replace this low low-degree influence requirement with
  the simpler low influence requirement follows by noting that
  \begin{equation}
    \Inf^{\le d}_i(f_j)
    \le
    \Inf_i(f_j)
  \end{equation}
\end{proof}

\subsection{From Discrete to Continuous}

We have shown that the SSC implies the Plurality is Stablest
Conjecture. We now show that the reverse is also
true for $\rho \ge -\frac{1}{q-1}$,
thereby establishing the equivalence of the Plurality is Stablest
Conjecture and
the SSC for $\rho$ in this range.

\begin{theorem}
  \label{thm:reverseTransf}
  For any $q\ge 2$, $n\ge 1$, $\rho \in [-\frac{1}{q-1},1]$, $\epsilon>0$,
  $\tau>0$
  and $g:\Reals^n \rightarrow E_q$
  there exist an
  $m$ and an
  $f:[q]^m \rightarrow \Delta_q$
  with $\Inf_i(f_j) \le \tau, \forall i,j$
  such that $\E f = \E g$ and
  \begin{equation}
    \label{eq:discToContSp}
    \left|
      \Sp(f)
      -
      \Sp(g)
    \right|
    \le \epsilon
  \end{equation}
\end{theorem}
\begin{proof}
  Let $\tilde{g}=U_{1-\delta} g$ for some small $\delta>0$ be a smooth
  version of $g$.
  By Lemma \ref{lem:OrnsteinConvergence} we can pick $\delta$ small enough so that
  \begin{equation}
    \label{eq:discToCont1}
    \left|
      \Sp(g)
      -
      \Sp(\tilde g)
    \right|
    \le \frac{\epsilon}{2}
  \end{equation}
  Let $m=r n$ for some $r$ to be determined later
  and $\omega$ uniform on $[q]^m$ and $\lambda \in [q]^m$ selected
  according to \eqref{eq:defSpQ}.
  For $i=1\ldots n$, let
  \begin{equation}
    V_i(\omega) = \frac{q}{\sqrt{q-1}} \frac{1}{\sqrt{r}} \sum_{l=1}^r
    \left(1_{\{\omega_{(i-1)m+l}=1\}} - \frac{1}{q} \right)
  \end{equation}
  and define $\tilde{f}$ by
  $\tilde{f}(\omega) = \tilde{g} (V_i(\omega))$.
  By the central limit theorem
  $
  (V(\omega),V(\lambda)) \stackrel{\mathcal{D}}{\rightarrow}
  (X,Y) \text{ as } r
  \rightarrow \infty$
  where $X,Y$ has the distribution of Conjecture
  \ref{conj:2GaussQSplitIsop}.
  Since $\tilde g$ is bounded and continuous
  (Lemma \ref{lem:OrnsteinSmooth}) we have
  \begin{equation}
    \label{eq:discToCont2}
    \left|
      \Sp(\tilde f)
      -
      \Sp(\tilde{g})
    \right|
    =
    \left|
      \E\ip{\tilde{g}(V(\lambda))}{\tilde{g}(V(\omega))}
      -
      \E\ip{\tilde{g}(X)}{\tilde{g}(Y)}
    \right|
    \rightarrow 0
    \text{ as }
    r \rightarrow \infty
  \end{equation}
  $\tilde{f}$ might not have the same balance as $\tilde{g}$,
  but by changing its value on at most
  $\frac{1}{2} \sum_{j=1}^q \left| \E \tilde{f}_j - \tilde{g}_j \right|$ points
  we can create an $f$ such that $\E f = \E \tilde{g} \,\,\, (= \E g)$ and
  \begin{equation}
    \label{eq:discToCont3}
    \left|
      \Sp(\tilde f)
      -
      \Sp(f)
    \right|
    \le
    \sum_{j=1}^q \left| \E \tilde{f}_j - \E \tilde{g}_j \right|
    \rightarrow 0
    \text{ as }
    r \rightarrow \infty
  \end{equation}
  Picking $r$ large enough and combining
  \eqref{eq:discToCont1}, \eqref{eq:discToCont2} and \eqref{eq:discToCont3}
  gives \eqref{eq:discToContSp}.

  It remains to show that the influences of $f$ can be made small.
  But this follows from $\tilde{g}$ being Lipschitz; changing only one
  variable, say $\omega_{(i-1)m+l}$ cannot change $V(\omega)$ by more
  than $\frac{q}{\sqrt{q-1}}\frac{1}{\sqrt r}$, so
  \begin{equation}
    \Inf_{(i-1)m+l} f_j  = \E_\omega[ \Var_{\omega_{(i-1)m+l}} f_j] \le A
    \frac{q}{\sqrt{q-1}}\frac{1}{\sqrt r}
    \le \tau
  \end{equation}
  for $r$ large enough.
\end{proof}

\begin{proof}[Proof of Theorem \ref{thm:plurStabApplRev}]
  This follows by combining Theorem \ref{thm:reverseTransf} and Lemma
  \ref{lem:PLUReqConj} and letting the $\eps$ of
  Theorem \ref{thm:reverseTransf} go to $0$.
\end{proof}
\section{Conclusion}
In this paper we have demonstrated the relationship between optimally
noise stable low-influence partitions of discrete space and optimally noise
stable partitions of Gaussian space.
In particular we have applied this relationship to various problems in
social choice theory and hardness of approximation in computer science.

Of the two generalizations of Theorem~\ref{thm:Borell} considered
we have proved one. The other one remains an open problem.
We also note that the two directions of generalizations can be
combined yielding a more general conjecture stating that
the standard simplex partition of
$\Reals^n$ into $q>2$ parts maximizes the probability that $k>2$
positively correlated Gaussians fall into the same part.

It should be noted that our results give a direct correspondence
between discrete low-influence noise stability and Gaussian noise
stability. That is, even if the SSC is false, whatever the most stable
partition of Gaussian space is, it can be used to construct a most
stable low-influence balanced social choice function.
Moreover, as long as the \emph{least} stable partition of Gaussian
space does not depend on $\rho$ when $\rho \le 0$, it will give an
explicit optimal UGC hardness result for MAX-q-CUT.

\section*{Acknowledgments}
The second author would like to acknowledge Subhash Khot, Guy Kindler and Ryan O'Donnell for the introduction of the problem of Plurality is Stablest.
He would further like to thank Krzysztof Oleszkiewicz and Ryan O'Donnell for discussions related to SSC and to Krzysztof Oleszkiewicz for pointing out the relationship between
spherical stereometric results and Gaussian stability results. Finally the authors would like to thank Christer Borell, Johan H{\aa}stad, Guy Kindler and Jeffrey Steif
for helpful discussions and comments.

\bibliographystyle{plain}
\bibliography{condImpl,all,my}

\appendix
\section{Approximability of MAX-q-CUT}
\label{sec:maxqcut}

In this appendix we show that if we assume the Unique Games
Conjecture, then the optimal approximability constant of
MAX-q-CUT is directly related to the most stable partition of Gaussian
space into $q$ parts of equal measure as described in Conjecture
\ref{conj:2GaussQSplitIsop},
and together these two conjectures implies that the Frieze-Jerrum SDP
achieves the optimal approximation ratio.

\subsection{The Unique Games Conjecture}

The Unique Games Conjecture (UGC) was introduced by Khot in~\cite{Khot:02}.
It asserts the hardness of approximating the \emph{Unique Label
Cover} problem within any constant.

\begin{definition}
  The \emph{Unique Label Cover} problem,
  $\calL=(V,W,E,M,\{\sigma_{v,w}\}_{(v,w) \in E})$,
  is defined on a bipartite graph $(V \cup W, E)$ with a permutation
  $\sigma_{v,w} : [M] \rightarrow [M]$ associated with every edge
  $(v,w) \in E \subseteq V \times W$.
  A labeling $l:V\cup W \rightarrow [M]$ is said to satisfy an edge
  $(v,w)$ if
  \begin{equation}
    \sigma_{(v,w)}(l(w)) = l(v)
  \end{equation}
  The value of a labeling $l$, $\VAL_l (\calL)$, is the fraction of
  edges satisfied by $l$ and
  the value of $\calL$ is the maximal fraction of edges satisfied by any
  labeling,
  \begin{equation}
    \VAL(\calL) = \max_{l} \VAL_l(\calL)
  \end{equation}
\end{definition}

\begin{conjecture}
  \textbf{The Unique Games Conjecture.}
  For any $\eta, \gamma > 0$ there exists a $M=M(\eta,\gamma)$ such
  that it is NP-hard to distinguish instances $\calL$ of the Unique
  Label Cover problem with label set size $M$ having
  $\VAL(\calL) \ge 1-\eta$ from those having
  $\VAL(\calL) \le \gamma$.
\end{conjecture}

\subsection{Optimal  approximability constants}

Next, we will show that for any $\epsilon>0$,
MAX-q-CUT can be approximated within $\alpha_q-\epsilon$ in polynomial time
while it is UG-hard to approximate it within $\beta_q + \epsilon$
where the constants $\alpha_q$ and $\beta_q$ are given by,

\begin{definition}
  \label{def:optApxConstants}
  For $q \ge 1$, let
  \begin{equation}
    \label{eq:alpha_q}
    \alpha_q
    =
    \lim_{n \rightarrow \infty}
    \sup_{g}
    \inf_{-\frac{1}{q-1} \le \rho \le 1}
    \frac{q}{q-1}
    \frac{1-\Sp(g)}
    {1-\rho}
  \end{equation}
  and
  \begin{equation}
    \label{eq:beta_q}
    \beta_q
    =
    \lim_{n \rightarrow \infty}
    \inf_{-\frac{1}{q-1} \le \rho \le 1}
    \sup_{g}
    \frac{q}{q-1}
    \frac{1-\Sp(g)}
    {1-\rho}
  \end{equation}
  where
  the supremum is over all
  $g:\Reals^n \rightarrow E_q$.
\end{definition}
\noindent
Note that the limit in \eqref{eq:alpha_q} and \eqref{eq:beta_q}
exist since they are limits of bounded functions increasing with $n$
(we can always ignore any number of dimensions while
specifying the partition).

We now show that $\alpha_q = \beta_q$ assuming Conjecture
\ref{conj:2GaussQSplitIsop}. To do this, we first show that we can
restrict attention to non-positive values of $\rho$ and for all such
values the standard simplex partition is optimal.

\begin{lemma}
  \label{lem:posRhoDominated}
  Fix $g:\Reals^n \rightarrow E_q$. Then
  $
  \displaystyle
  \inf_{0 \le \rho \le 1}
  \frac{1-\Sp(g)}
  {1-\rho}
  $
  is obtained by $\rho=0$.
\end{lemma}

\begin{proof}
  Fix $\eps>0$. By Theorem \ref{thm:reverseTransf} there is an
  $f:[q]^{m(\eps)} \rightarrow \Delta_q$ with $\E g = \E f$
  such that
  \begin{equation}
    \Sp(g)
    \le
    \Sp(f) + \eps
  \end{equation}
  On the other hand, by \eqref{eq:SpFourier} and \eqref{eq:fourierExpressions}
  \begin{equation*}
    \Sp(f)
    =
    \sum_{\sigma} \rho^{|\sigma|} \|c_\sigma\|_2^2
    \le
    \|\E f\|_2^2 + \rho (\|f\|_2^2 - \|\E f\|_2^2)
    \le
    \rho
    + (1-\rho) \|\E f\|_2^2
  \end{equation*}
  Hence,
  \begin{equation}
    \frac{1-\Sp(f)}{1-\rho}
    \ge
    1 - \|\E f\|_2^2
  \end{equation}
  and by the construction of $f$,
  \begin{equation}
    \label{eq:posRhoDominated1}
    \frac{1-\Sp(g)+\eps}{1-\rho}
    \ge
    1 - \|\E g\|_2^2
  \end{equation}
  Letting $\eps \rightarrow 0$,
  and noting that \eqref{eq:posRhoDominated1} holds with equality for
  $\eps=\rho=0$ gives the result.
\end{proof}

\begin{theorem}
  Assume Conjecture \ref{conj:2GaussQSplitIsop}.
  Then $\alpha_q = \beta_q$.
\end{theorem}

\begin{proof}
  By Lemma \ref{lem:posRhoDominated} the infimums in the definition of
  $\alpha_q$ and $\beta_q$ are obtained for
  $-\frac{1}{q-1} \le \rho \le 0$.
  The result now follows from the fact that
  for $\rho$ in this range,
  the least stable partition in Conjecture
  \ref{conj:2GaussQSplitIsop} does not depend on $\rho$.
\end{proof}

We now proceed to present the approximation algorithm and the
inapproximability argument which together implies Theorem~\ref{thm:condMaxqCutResults}.

\subsection{An approximation algorithm}
\label{sec:apxAlgo}
The approximation algorithm presented here is a generalization of
the algorithm presented in~\cite{FriezeJerrum:95} allowing for an arbitrary
partition to be used when rounding the relaxed
solution. The algorithm in~\cite{FriezeJerrum:95} corresponds exactly to
using the simplex partition of Conjecture
\ref{conj:2GaussQSplitIsop}, which (as we will see) is optimal if Conjecture
\ref{conj:2GaussQSplitIsop} is true.

Let
$a_1, \ldots, a_q \in \Reals^{q-1}$ be generating vectors of a standard simplex
partition of $\Reals^{q-1}$, i.e. satisfying \eqref{eq:stdSimplexVecConstr}.
Labeling the vertices with vectors from $a_1, \ldots a_q$ instead of numbers
from $[q]$, we can write the value of a MAX-q-CUT instance $\calM_q(V,E,w)$ as the following discrete optimization problem:

\begin{equation}
   \VAL(\calM_q)
   =
  \begin{array}[t]{ll}
    \max
    & \frac{q-1}{q} \sum_{(u,v) \in E} w_{(u,v)} (1 - l_u \cdot
    l_v)
    \vspace{2mm}
    \\
    \mbox{subject to} &
    l_u \in \{a_1, \ldots, a_q\} \mbox{ , } \forall u \in V
  \end{array}
\end{equation}

To obtain the SDP relaxation we allow the vectors to be arbitrary
points on the unit sphere $S^{n-1}$ while adding the constraint
$z_u \cdot z_v \ge -\frac{1}{q-1}$
which by \eqref{eq:stdSimplexVecConstr} holds for
vectors in $\{a_1, \ldots, a_q\}$,

\begin{equation}
  \mbox{SDP-VAL} (\calM_q)
  :=
  \begin{array}[t]{ll}
  \max &
  \frac{q-1}{q} \sum_{(u,v) \in E} w_{(u,v)} (1 - z_u \cdot z_v)
  \vspace{2mm}
  \\
  \mbox{subject to} &
  \begin{array}[t]{l}
    z_u \in \Reals^n { , } \forall u \in V
    \\
    z_u \cdot z_u = 1 { , } \forall u \in V
    \\
    z_u \cdot z_v \ge -\frac{1}{q-1} { , } \forall u,v \in V
  \end{array}
\end{array}
\end{equation}
where $n=|V|$ denotes the number of vertices.

The rounding applied to the solution of SDP-VAL is parametrized by an
integer $m$, a
partition $\cal{A}$ = $\{A_1, \dots, A_q\}$ of $\Reals^{m}$ and
an error constant $\delta>0$,

\newtheorem*{apxAlgo}{Approximation algorithm
  $\mathcal{R}({m, \cal{A}, \delta)}$}
\begin{apxAlgo}{\ }
  \begin{enumerate}
  \item Compute an almost optimal solution $(z_u)_{u \in V}$ to
    $\mbox{SDP-VAL}(\mathcal{M}_q)$ using semidefinite programming.
    This will achieve a value of $\mbox{SDP-VAL}(\mathcal{M}_q) -
    \delta$.
  \item
    Pick a projection matrix $T:\Reals^{m \times n}$,
    by letting $T_{ij}$ be i.i.d. $\Norm(0,1)$.

  \item
    For each $u \in V$, let $l(u) = i$ iff $T z_u \in A_i$.
  \end{enumerate}
\end{apxAlgo}

Let $\mbox{R-VAL}(\calM_q) = \VAL_l(\calM_q)$ be the value of
the rounded labeling. Then, the expected approximation ratio is:

\begin{eqnarray*}
  \frac{\E[\mbox{R-VAL}(\calM_q)]}{\VAL(\calM_q)}
  \ge
  \frac{\E[\mbox{R-VAL}(\calM_q)]}{\mbox{SDP-VAL}(\calM_q) + \delta}
  =
  \frac{
    \sum_{(u,v)\in E} w_{(u,v)}
    \P\left(l(u) \neq l(v)\right)
  }{
    \frac{q-1}{q} \sum_{(u,v) \in E} w_{(u,v)} (1 - z_u \cdot z_v) + \delta
  }
  \ge
  \\
  \ge
  \frac{q}{q-1} \inf_{\stackrel{z_u,z_v \in S^{n-1}}{z_u \cdot z_v \ge -\frac{1}{q-1}}}
  \frac{1-\P((Tz_u,Tz_v) \in A_1^2\cup \dots \cup A_q^2)}
  {1-z_u \cdot z_v + \delta}
\end{eqnarray*}
But,
$Tz_u, Tz_v \in \Norm(0,I_m)$ and $\Cov(Tz_u, Tz_v) = (z_u \cdot z_v) I_m$,
so by picking $m$ large enough and $A_1, \dots, A_q$ so that the limit in
\eqref{eq:alpha_q} is almost achieved (bar, say $\delta$), and then picking
$\delta=\delta(\epsilon)$ small enough,
we get an approximation ratio of $\alpha_q - \epsilon$, for any $\epsilon>0$.
We have proved the following result
\begin{theorem}
  For any $\epsilon>0$ there exists a polynomial time
  algorithm that approximates MAX-q-CUT within
  $\alpha_q\!-\!\epsilon$.
\end{theorem}

\subsection{Inapproximability results}

We will now prove that MAX-q-CUT is UG-hard to approximate within any
factor greater than $\beta_q$. To do so, we present a reduction from
the Unique Label Cover problem to MAX-q-CUT following the same outline
as the corresponding reduction for MAX-CUT given in~\cite{KKMO:07}.
The reduction is based on a \emph{Probabilistically Checkable Proof} (PCP)
whose proof $\Pi$ consists of the function tables of $\{f_w\}_{w\in W}$, where
$f_w : [q]^M \rightarrow [q]$ is expected to be the \emph{long code}
of $w$'s label $l(w)$, i.e. $f_w(x) = x_{l(w)}$.
In order to be able to reduce the PCP to MAX-q-CUT, the PCP verifier
$\calV_\rho$
is designed to use an acceptance predicate which reads two random
function values from the proof and accepts iff they differ.
Thus, a MAX-q-CUT instance $\calM_q$ can be created from the PCP
by letting the vertices be the function values that can be read by
$\calV_\rho$, the edges the pairs of function values that are
compared, and the weights the probability of that comparison being
made by $\calV_\rho$.
The verifier is parametrized by $\rho \in [-\frac{1}{q-1},1]$.

\newtheorem*{pcpVerifier}{PCP Verifier $\calV_\rho$}
\begin{pcpVerifier}{\ }
  \begin{enumerate}
  \item Pick $v \in V$ at random and two of its neighbors $w,w'$ at
    random.
  \item Pick $x \in [q]^M$ at random.
  \item Pick $y \in [q]^M$ to be a $\rho$-correlated copy of $x$,
    i.e. each $y_i$ is independently selected using the conditional
    distribution
    \begin{equation}
      \mu(y_i|x_i)=\rho 1_{\{y_i=x_i\}} + (1-\rho) \frac{1}{q}
    \end{equation}
  \item Accept if $f_wP_{\sigma_{v,w}}(x) \neq
    f_{w'}P_{\sigma_{v,w'}} (y)$
  \end{enumerate}
  where $P_\sigma : [q]^M \rightarrow [q]^M$ denotes the function
  $P_\sigma(x_1, \dots, x_M) = (x_{\sigma(1)}, \dots, x_{\sigma(M)})$.
\end{pcpVerifier}

Using a result from~\cite{KhotRegev:03} we can assume that the graph is regular
on the $V$ side so that $(v,w)$, and similarly $(v,w')$,
picked by $\calV_\rho$ corresponds to a
an edge selected uniformly at random.

\begin{lemma}{\em\textbf{(Completeness). }}
  \label{lem:completeness}
  Fix $\rho \in [-\frac{1}{q-1},1]$.
  Then, for any Unique Label Cover problem $\calL$ with
  $\VAL(\calL) \ge 1-\eta$ there exists a proof $\Pi$ such that
  \begin{equation}
    \P[\calV_\rho \text{ accepts } \Pi]
    \ge
    (1-2\eta) \frac{q-1}{q}(1-\rho)
  \end{equation}
\end{lemma}

\begin{proof}
  Let $l$ be the optimal assignment for $\calL$ and
  $f_w$ be the long code of $l(w)$, i.e.
  \begin{equation}
    f_w(x) = x_{l(w)}
  \end{equation}

  With probability at least $1-2\eta$, both edges $(v,w)$ and $(v,w')$
  are satisfied by $l$. In this case,
  \begin{equation}
    f_wP_{\sigma_{v,w}}(x) = x_{\sigma_{v,w}(l(w))} = x_{l(v)}
    \text{ and }
    f_{w'}P_{\sigma_{v,w'}}(y)
    = y_{l(v)}
  \end{equation}
  and $\calV_\rho$ accepts with probability
  \begin{equation}
    \P[x_{l(v)} \neq y_{l(v)}]
    =
    1 - \left(\rho + \frac{1-\rho}{q}\right)
    =
    \frac{q-1}{q}(1-\rho)
  \end{equation}
\end{proof}

\begin{lemma}{\em\textbf{(Soundness).} }
  \label{lem:soundness}
  Fix $\rho \in [-\frac{1}{q-1},1]$ and $\epsilon>0$.
  Then, there exists a $\gamma=\gamma(q,\rho,\epsilon)>0$
  such that for any Unique Label Cover problem $\calL$ with
  $\VAL(\calL) \le \gamma$ and any proof $\Pi$,
  \begin{equation}
    \label{eq:soundness}
    \P[\calV_\rho \text{ accepts } \Pi]
    \le
    1 - \Lambda^-_q(\rho) + \epsilon
  \end{equation}
  where
  \begin{equation}
    \label{eq:lambda_q_minus}
    \Lambda^-_q(\rho) = \lim_{n \rightarrow \infty}
    \inf_{g:\Reals^n \rightarrow E_q} \Sp(g)
  \end{equation}
\end{lemma}

\begin{proof}
  For $w \in W$, let $\tilde{f}_w : [q]^M \rightarrow E_q$ defined
  by
  \begin{equation}
    \tilde{f}_w(x) = \vec{e}_{f_w(x)}
  \end{equation}
  map the value of $f_w$ onto one of $q$ unit vectors,
  and for $v \in V$, let $g_v: [q]^M \rightarrow \Delta_q$ be defined by
  \begin{equation}
    g_v(x) = \E_{w}[\tilde{f}_w P_{\sigma_{v,w}}(x)]
  \end{equation}
  where the expectation is over a random neighbor $w$ of $v$. Then,
  \begin{eqnarray*}
    \P[\calV_\rho \text{ accepts } \Pi]
    &=&
    \E_{v,w,w',x,y}[1-\ip{\tilde{f}_w
      P_{\sigma_{v,w}}(x)}{\tilde{f}_{w'}P_{\sigma_{v,w'}}(y)}]
    =
    \\
    &=&
    1-\E_{v,x,y}
    \left[
      \ip
      {g_v(x)}
      {g_v(y)}
    \right]
    =
    1-\E_v \Sp(g_v)
  \end{eqnarray*}

  Now suppose $\Pi$ is a proof such that \eqref{eq:soundness} is not
  satisfied, i.e,
  \begin{equation}
    \label{eq:soundRev}
    \E_v \Sp(g_v)
    <
    \Lambda^-_q(\rho) - \epsilon
  \end{equation}

  We need to show that this implies $\VAL(\calL) > \gamma$.
  To do so it is enough to create a random labeling $l$ such that
  \begin{equation}
    \E_l [\VAL_l(\calL)] > \gamma
  \end{equation}

  Let $V_{\text{good}} = \{v \in V|\Sp(g_v) \le \Lambda^-_q(\rho) -
  \frac{\epsilon}{2}\}$.
  Since $\Sp (g_v) \ge 0$, \eqref{eq:soundRev} implies that
  $|V_{\text{good}}| \ge \frac \epsilon 2 |V|$.
  Further, for $v \in V_{\text{good}}$,
  Theorem \ref{thm:lowInflStabToGaussStab} implies that
  $\max_i \Infd_i g_v \ge \tau$, for some $d$ and $\tau>0$
  depending only on $q$,$\rho$ and $\epsilon$.

  The assignment $l$ is created as follows:
  \begin{enumerate}
  \item For $v \in V$, let $l(v)=i$, where $i$ maximizes $\Infd_i g_v$ (ties broken arbitrarily)
  \item For $w \in W$, let $l(w)=i$ with probability proportional to
    $\Infd_i \tilde{f}_w$.
  \end{enumerate}
  Since \eqref{eq:dInflBound} holds for vector-valued functions,
  this means that
  \begin{equation}
    \P_l(l(w) = i)
    \ge
    \frac{\Infd_i \tilde{f}_w}{q d}
  \end{equation}
  For $v \in V_{\text{good}}$,
  \begin{eqnarray*}
    \tau
    &\le&
    \Infd_{l(v)} g_v
    =
    \Infd_{l(v)} \E_{w}[\tilde{f}_w P_{\sigma_{v,w}}(x)]
    \le
    \E_{w} \Infd_{l(v)} \tilde{f}_w P_{\sigma_{v,w}}(x)
    =
    \\
    &=&
    \E_{w} \Infd_{\sigma_{v,w}^{-1}(l(v))} \tilde{f}_w (x)
    \le
    qd \P_{w,l}[l(w) = \sigma_{v,w}^{-1}(l(v))]
    =
    qd \P_{w,l}[l \text{ satisfies } (v,w)]
  \end{eqnarray*}
  where the second inequality follows from convexity of $\Infd_i$.
  Hence,
  \begin{equation}
    \E_l [\VAL_l(\calL)]
    =
    \P_{l,v,w}(l \text{ satisfies } (v,w))
    \ge
    \frac \epsilon 2
    \cdot
    \frac \tau {qd}
  \end{equation}
  Picking $\gamma = \frac \epsilon 4 \cdot \frac \tau {qd} > 0$
  finishes the proof.
\end{proof}

Together, the soundness and completeness lemmas implies the following
inapproximability result for MAX-q-CUT:

\begin{theorem}
  For any $\epsilon>0$ it is UG-hard to approximate MAX-q-CUT within
  $\beta_q + \epsilon$.
\end{theorem}
\begin{proof}
  By Lemma \ref{lem:completeness} and \ref{lem:soundness} it is
  UG-hard to distinguish instances of MAX-q-CUT with value
  at least
  $(1-2\eta) \frac{q-1}{q}(1-\rho)$
  from instances with value at most
  $1 - \Lambda^-_q(\rho) + \epsilon$
  for any $\eta,\epsilon>0$.
  Thus, it is UG-hard to approximate MAX-q-CUT within
  \begin{equation}
    \frac
    {1 - \Lambda^-_q(\rho) + \epsilon}
    {(1-2\eta) \frac{q-1}{q}(1-\rho)}
    =
    \frac{q}{q-1}
    \frac
    {1 - \Lambda^-_q(\rho)}
    {1-\rho}
    + \epsilon'
  \end{equation}
  where $\epsilon'>0$ can be made arbitrarily small by picking
  $\eta$ and $\epsilon$ small enough.
  Since this holds for any $\rho \in [-\frac{1}{q-1},1]$ the result follows.
\end{proof}

\end{document}